\documentclass[reqno,10pt]{amsart} 
\usepackage{amsmath,amsthm} 
\usepackage{amssymb}
\usepackage{hyperref} 
\usepackage{mathabx}
\newtheorem{Theorem}{Theorem} 
\newtheorem{lemma}[Theorem]{Lemma} 
\newtheorem{remark}[Theorem]{Remark}

\numberwithin{equation}{section}
\font\ff=cmsy10 
\def\tiF{\text{\ff F\kern 0pt}{\;}^{ -1}} 
\def\tF{\text{\ff F\kern 0pt}} 

\begin{document} 
\title[]{Unique continuation principle for high order\\ equations of Korteweg-de Vries type} 
\author{Pedro Isaza}  
\thanks{With the support of  DIME, Universidad Nacional de Colombia, Sede Medell'n} 
  
\address{Pedro Isaza J. \newline  
Departamento de Matem\'aticas\\Universidad Nacional de Colombia\newline  
A. A. 3840 Medell\'{\i}n, Colombia} 
\email{pisaza@unal.edu.co}
\newcommand{\ldos}{{L^2(\mathbb R)}}
\newcommand{\h}[1]{{H^{#1}(\mathbb R)}}
\newcommand{\eb}{e^{\beta x}}
\newcommand{\el}{e^{\lambda x}}
\newcommand{\liTldx}{{{}_{L_T^\infty L_x^2}}}
\newcommand{\lixldT}{{{}_{L^\infty_x L^2_T}}}
\newcommand{\luxldT}{{{}_{L^1_x L^2_T}}}
\newcommand{\ldxliT}{{{}_{L^2_x L^\infty_T}}}
\newcommand{\ldxldT}{{{}_{L^2_x L^2_T}}}
\newcommand{\ldTldx}{{{}_{L^2_T L^2_x}}}
\newcommand{\ldTlix}{{{}_{L^2_T L^\infty_x}}}
\newcommand{\luTldx}{{{}_{L^1_T L^2_x}}}
\newcommand{\litldx}{{{}_{L_t^\infty L_x^2}}}
\newcommand{\lixldt}{{{}_{L^\infty_x L^2_t}}}
\newcommand{\luxldt}{{{}_{L^1_x L^2_t}}}
\newcommand{\lutldx}{{{}_{L^1_t L^2_x}}}
\newcommand{\ldosD}{{L^2(D)}}
\newcommand{\R}{{\mathbb R}}
\newcommand{\vp}{{\varphi}}
\newcommand{\e}{{\varepsilon}}
\newcommand{\N}{{\scriptscriptstyle{N}}}
\newtheorem*{TI}{Theorem I}
\newtheorem*{TII}{Theorem II}

\begin{abstract} 
In this article we consider the problem of  unique continuation a class of high order equations of Korteweg-de Vries type
which include the kdV hierachy. It is proved that if the difference $w$ of two solutions  of an equation of this form has certain exponential decay for $x>0$ at two different times, then $w$ is identically zero.

\end{abstract} 
\maketitle

\section{Introduction}
This article is concerned with a unique continuation principle for the equation\begin{equation} 
\partial _{t}u+ (-1)^{k+1}\partial^ {n}_{x}u+P(u, \partial _{x}u, \cdots, \partial_x^{p} u) =0\,,\quad u=u(x,t),\quad x,t\in\mathbb R,\label{ec}
\end{equation}where $n=2k+1$, $k=1,2\cdots$, and $P$ is a polynomial in $u, \partial_xu,\cdots,\partial^{p}_xu$, with $p\leq{n-1}$.
In particular, we will focus our attention to the case in which  $P$ has the form 
\begin{equation} P(u, \partial _{x}u, \cdots, \partial_x^{n-2} u)=\sum_{d=2}^{k+1}\kern-50pt\sum_{\kern 50pt |{ m}|=2(k+1-d)+1}\kern-50pt a_{d,m}\,\partial_x^{m_1}u\cdots\partial_x^{m_d}u\equiv\sum_{d=2}^{k+1}A_d(z),\quad z=(u,\partial_xu,\cdots,\partial_x^{n-2}u),\label{poli1}
\end{equation}
where,  for $d\in\mathbb N$  and  for integers $m_1,\cdots,m_d$, ${m}:=(m_1,\cdots,m_d)$ is a multiindex with $0\leq m_1\leq\cdots\leq m_d$, $|{ m}|:=m_1+\cdots +m_d$, and $a_{d,m}$ is a constant. We will refer to equation \eqref{ec} with $P$ as in \eqref{poli1} as equation \eqref{ec}-\eqref{poli1}. Besides, we will  also be considering equation \eqref{ec} when the nonlinearity $P$ has order $p\leq k$.

The type of relation expressed in \eqref{poli1}, between the degree and the order of each monomial of $P$,  is present in the nonlinearities of the collection of equations known as the KdV (Korteweg-de Vries) hierarchy. This set of equations was introduced by Lax in \cite{L1}  in the process 
to determine the functions $u=u(x,t)$ for which the eigenvalues of the operator $L:=\frac{d^2}{dx^2}-u(\cdot,t)$ remain constant as $t$ evolves. This property had been already discovered  by Gardner et al. in \cite{G1}  for the solutions of the Korteweg-de Vries equation
$$\partial_tu+\partial^3_xu+u\partial_xu=0.$$
 Lax showed that this property holds for the solutions of the equations
\begin{equation}\partial_tu+[B_k(u),L]=0,\label{kdvH}\end{equation}
where $[B,L]:=BL-LB$ denotes the commutator of $B$ and $L$, and $B_k(u)$ is the skew-adjoint operator defined by
$$B_k(u)=b_k\frac{d^{2k+1}}{dx^{2k+1}}+\sum_{j=0}^{k-1}b_{k,j}(u)\frac{d^{2j+1}}{dx^{2j+1}} +\frac{d^{2j+1}}{dx^{2j+1}}b_{k,j}(u),$$
with the coefficients $b_{k,j}(u)$ chosen in such a way that the operator $[B_k(u),L]$ has order zero. It was proved in \cite{L2}
that the equations in the KdV hierarchy \eqref{kdvH} can be written in the form $\partial_tu+\partial_xG_{k+1}(u)=0$, were the functions $G_k(u)$ are the gradients of the functionals $F_k(u)$ which define the conservation laws of the KdV equation. The gradients $G_k$ satisfy the following recursion formula due to Lenard (see \cite{G2} and \cite{P}):
$$\partial_xG_{k+1}=cJG_k\,,\quad\text{where } J=\partial_x^3+\frac23u\partial_x+\frac13 \partial_xu\,.$$
This formula can be applied  to obtain a derivation of the  equations in the hierarchy. Starting with $G_0(u)=3$, with $k=0$ we get the transport equation, with $k=1$ the KdV equation, and, with $k=2$,  $k=3$, and $k=4$,  we respectively find  the equations

\begin{equation}\partial_tu+\partial_x^5u-10u\partial_x^3u-20\partial_xu\partial^2_xu+30u^2\partial_xu=0,\label{fifthorder}\end{equation}

$$\partial_tu+\partial_x^7u+14u\partial_x^5u+42\partial_xu\partial^4_xu+70\partial_x^2u\partial_x^3u+70u^2\partial^3_xu+280u\partial_xu\partial^2_xu+70(\partial_xu)^3+140u^3\partial_xu=0.$$

$$\partial_tu+\partial^9_xu +\sum_{\substack{m_1+m_2=7\\0\leq m_1\leq m_2}}\kern-5pt a_{2,m}\,\partial^{m_1}_xu\partial_x^{m_2}u+ \sum_{\substack{m_1+m_2+m_3=5\\0\leq m_1\leq m_2\leq m_3}} \kern-10pt a_{3,m}\,\partial^{m_1}_xu\partial_x^{m_2}u\partial_x^{m_3}u+\cdots +a_{5,m}u^4\partial_xu=0,$$
for certain constants $a_{d,m}$, with $d=2,\cdots, 5$ and  $|m|=2(5-d)+1$.

In spite of  computational difficulties, it is possible to obtain exact expressions for all the equations in the hierachy (see \cite{A}). However, following a simple procedure, and without obtaining the explicit values for the coefficients, it can be proved  (see \cite{Gr}) that the  equations in the KdV hierachy \eqref{kdvH} have the form of \eqref{ec}-\eqref{poli1}.  When $k$ is even we have made the change of variable $x\mapsto -x$ and thus the linear term $\partial^{n}_xu$ has been transformed into  $(-1)^{k+1}\partial^{n}_xu$ in \eqref{ec}. 

The aspects of local and global well-posedness of the initial value problem (IVP) associated with the general equation \eqref{ec} have been considered in  \cite{KPV3} and \cite{KPV4}, where Kenig, Ponce, and Vega proved that the (IVP) is  locally well-posed in  weighted spaces 
$H^s(\mathbb R)\cap L^2(|x|^m\,dx)$ if $s\geq s_0(k)$, for some $s_0(k)$ and some integer $m=m(k)$. 

For the (IVP) associated to  \eqref{ec}-\eqref{poli1}, in \cite{S}, Saut proved the existence of global solutions  for initial data in Sobolev spaces $H^m(\mathbb R)$ for $m\geq k$,  integer. By using a variant of Bourgain
spaces, in \cite{Gr}, Gr\"unrok proved the local well-posedness for the (IVP) of equation \eqref{ec}-\eqref{poli1} in the context of  the spaces 
$$\widehat {H}_{s}^r(\mathbb R)
:=\{f\mid \|f\|_{s,r}:=\|(1+\xi^2)^{s/2}\widehat{f}(\xi)\|_{L_{\xi}^{r'}}<\infty\},\quad \text{with } r\in(1,{\textstyle{\frac{2k}{2k-1}}}], \; \text{and }s>k-{\textstyle\frac32-\frac1{2k}+\frac{2k-1}{2r'}}.$$
Here $\widehat{\;}$ denotes the Fourier transform and $1/r+1/r'=1$.
We also refer to the articles \cite{Po}, \cite{Li}, \cite{Pi},  \cite{Kw}, which, among others, consider the problem of well-posedness for high order equations of KdV-type and especially for the equations of order five ($k=2$).

Our main goal is to prove  continuation principles  for the equations  \eqref{ec}-\eqref{poli1} with $n\geq 5$, which include the KdV hierarchy,  and for the equations \eqref{ec} with $n\geq 5$ and $p\leq k$. Roughly speaking, we will prove that if the difference $w:=u_1-u_2$ of two sufficiently smooth solutions of  equation \eqref{ec}-\eqref{poli1} decays as $exp({-x_+^{4/3^+}})$ at two different times, then $w\equiv 0$. (Here $x_+:=\frac12({x+|x|})$, and $4/3^+$ means $4/3+\epsilon$ for arbitrarily small $\epsilon>0$). For \eqref{ec} with $p\leq k$ we have a similar result if $w$ decays as $exp({-ax_+^{n/(n-1)}})$ for  $a>0$ sufficiently large at two different times. 
This last result is coherent with the decay $exp({-cx_+^{n/(n-1)}})$ of the fundamental solution of the linear problem associated with equation \eqref{ec} (see \cite{SSS}). When the nonlineariy $P$ has higher order as in \eqref{ec}-\eqref{poli1}, it is then necessary to impose a stronger decay on $w$.
 
 The aspect of unique continuation  has been studied for a variety of non-linear dispersive equations, and especially for the KdV and Schr\"odinger equations. In \cite{SS}, Saut and Sheurer considered a class of nonlinear dispersive equations, which  includes the KdV equation, and proved that if a solution $u$ of one of such equations vanishes in an open set $\Omega$ of the space-time space, then  $u$ vanishes in all horizontal components of $\Omega$, that is, in the set $\{(x,t)\mid \exists\, y \;\text{with } (y,t)\in\Omega\}$. 
 
 By using methods of complex analysis, in \cite{B}, Bourgain proved that if a solution $u$ of the KdV equation is supported in a compact set $\{(x,t)\mid -B\leq x\leq B,\; t_0\leq t\leq t_1\}$, then $u$ vanishes identically.  
 
 In \cite{KPV1}, Kenig, Ponce and Vega, considered a solution of the KdV equation which vanishes only in two half lines $\mathbb [B,+\infty)\times\{t_0\}$  and $\mathbb [B,+\infty)\times\{t_1\}$, and proved that this solution must be identically zero. A similar result was proved in \cite{KPV2} for the difference $w=u_1-u_2$ of two solutions of the KdV equation. In \cite{EKPV1}, Escauriaza et al.  refined this result by only imposing the condition that  $w(\cdot,t_0)$ and $w(\cdot,t_1)$ decay as ${exp}({-ax_+^{\gamma}})$, for $\gamma=3/2$ and $a>0$ sufficiently large, together with a suplementary hypothesis of polynomial decay for $u_1$ and $u_2$. This result is obtained by  applying two types of estimates for the function $w$: Carleman type estimates, which express a boundedness of the inverse of the linear operator $\partial_t+\partial_x^3$  in $L^p-L^q$-spaces with exponential weight; and a so-called lower estimate which bounds the $L^2$-norm of $w$ in a small rectangle at the origin with the $H^2$-norm of $w$ in a distant 
 rectangle $[R,R+1]\times[0,1]$.
  
For the fifth order equation \eqref{ec} ($k=2$), in \cite{D}, Dawson proved a result similar to that in \cite{EKPV1} with $\gamma=4/3^+$ for the general case $p\leq n-1=4$, and with $\gamma=5/4$ for the case $p\leq 2$.

In this article we consider  equations \eqref{ec} and \eqref{ec}-\eqref{poli1} with arbitrary order $n$   and prove the continuation principles stated in Theorems I and II below. For that, we follow the method traced  in \cite{EKPV1}.  The  greatest difficulty in this process is to manage  the huge amount of terms arising in  the computations  of the operators involved in the lower estimate. We consider that the main contribution of our work is the presentation of  a clear and organized procedure to obtain the lower estimate (see Lemma \ref{lema3} and Theorem \ref{estinferior}). 

We now state our main results:
\begin{TI}\label{principal}
For  odd $n\geq 5$ ,  $k=(n-1)/2$,  and $\alpha>\frac{n+1}3$, suppose that $u_1,\,u_2\in C([0,1];H^{n+1}(\mathbb R)\cap L^2((1+x_+)^{2\alpha}\,dx))$ are two solutions of  the equation
\begin{equation}\partial _{t}u+ (-1)^{k+1}\partial^ {n}_{x}u+P(u, \partial _{x}u, \cdots, \partial_x^{n-2} u) =0\label{Lax}\end{equation}
with $P$ as in \eqref{poli1},
and let $w:=u_1-u_2$.  If
\begin{equation}w(0), w(1)\in L^2(e^{2x_+^{4/3+\epsilon}}dx)\label{w}\end{equation}
for some $\epsilon>0$, then $w\equiv0$.
\end{TI}
 The proof of Theorem I can be adapted to obtain a similar continuation principle for equation \eqref{ec} when  $p\leq k$. In this case we require a weaker decay for $w(0)$ and $w(1)$ and consider some minor modifications in  the  polynomial decay hypothesis for $u_1$ and $u_2$.    For the sake of simplicity we state this result without making special emphasis in the optimal value of  $\alpha$.

\begin{TII} For  odd $n\geq 5$, $k=(n-1)/2$,  and $\alpha_0>0$ sufficiently large, suppose that $u_1,\,u_2\in C([0,1];H^{n+1}(\mathbb R)\cap L^2((1+|x|)^{2\alpha_0}\,dx))$ are two solutions of  the equation
\begin{equation}\partial _{t}u+ (-1)^{k+1}\partial^ {n}_{x}u+P(u, \partial _{x}u, \cdots, \partial_x^p u) =0\label{Lax1}\end{equation}
with $p\leq k$. Define $w:=u_1-u_2$.  Then,  there is $a>0$ ,which depends only on $n$, such that if 
\begin{equation}w(0), w(1)\in L^2(e^{2ax_+^{n/{(n-1)}}}dx),\label{decayhyp2}\end{equation}
then $w\equiv0$.
\end{TII}
The article is organized as follows: In section \ref{exponen} we prove that the exponential decay for $w$ in the  semi-axis $x>0$  is preserved  in time. In section \ref{sectioncar} we establish the Carleman type estimates and in section \ref{sectionlow} we prove the lower estimates. Finally we give the proofs of Theorem I and Theorem II in section \ref{sectionmain}.

 Throughout the paper the letters $C$ and $c$ will denote diverse positive constants which may change from line to line, and whose dependence on certain parameters is clearly established in all cases. Sometimes, for a parameter $a$, we will use the notations $C_a$, $C(a)$, and $c_a$ to make emphasis in the fact that the constants  depend upon $a$.  We frequently write $f(\cdot_s)$ to denote a function $s\mapsto f(s)$. For a set $A$, $\chi_A$ will denote the characteristic function of $A$. The symbols $\widehat{\;}$ and $\widecheck{\;}$ will denote the Fourier and the inverse Fourier transform, respectively. The notations $\,\widehat{\;}^{\,\,{}_{x}}$ and  $\widecheck{\;}^{\,{}_{\xi}}$ will emphasize the facts that the Fourier transform and its inverse are taken with respect to  specific variables $x$ and $\xi$, respectively. For $1\leq p,q<\infty$, $A,B\subseteq\mathbb R$, $D=A\times B$, and $f=f(x,t)$ we will denote
 $$\|f\|^p_{L^p_xL^q_t(D)}:=\int_A\Bigl(\int_B|f(x,t)|^q\,dt\Bigr)^{p/q}\,dx\,.$$
 We will  use  similar definitions when  $p=\infty$ or $q=\infty$ and also for $\|f\|_{L^q_tL^p_x(D)}$.

\section{Exponential decay}\label{exponen}

In this section we prove that if the difference $w$ of two solutions of \eqref{Lax} decays exponentialy at $t=0$, then this decay is preserved at all positive times. This property will be crucial for the application of the Carleman estimates in the proofs of Theorems I and II. 
\begin{Theorem}\label{decaimiento}
For odd $n\geq 5$, $k=(n-1)/2$, and $\alpha>(n+1)/4$, let $u_1,u_2\in C([0,1];\h{n+1}\cap L^2((1+x_+)^{2\alpha}\,dx)))$ be  two solutions of \eqref{ec}-\eqref{poli1}, and define $w:=u_1-u_2$. Let $\beta>0$ and suppose that $w(0)\in L^2(e^{\beta x} \, dx)$. Then
\begin{equation}
\sup_{t\in[0,1]}\|w(t)\|_{L^2(e^{\beta x} \, dx)}<\infty.\end{equation}
\end{Theorem}
\begin{proof}
Let us denote $z_i=(u_i,\partial_xu_i,\cdots\partial_x^{n-2}u_i)$, $i=1,2$. Then $w$ is a solution of the differential equation
\begin{equation}
\partial _{t}w+ (-1)^{k+1}\partial^ {n}_{x}w+ P(z_1)-P(z_2)=0.\label{apr}
\end{equation}

We will first prove that the theorem is valid provided we can construct a  sequence $\{\phi_\N \}_{\N \in\mathbb N}$ of  nondecrasing functions in $ C^{\infty}(\R)$ satisfying  for all $x\in\mathbb R$ the conditions
\begin{equation}\varphi_\N (x)\to e^{\beta x}\quad \text{as }\, N\to\infty\quad \text{and}\quad 0\leq \varphi_\N (x)\leq C\,e^{\beta x},\label{condfi3}\end{equation}
\begin{equation} \varphi_\N (x)\leq C_\N (1+x_+)^{(k+2)/4}\,,\label{condfi2}\end{equation}
\begin{equation}|\varphi_\N ^{(j)}(x)|\leq C_j\varphi'_\N (x)\quad\text{for }j=2,3\cdots,n=2k+1\;\;\text{and } \varphi_\N'(x)\leq  C\varphi_\N(x),\label{condfi}
\end{equation}
\begin{equation}
\varphi_\N (x)\leq C(1+x_+)\varphi'_\N (x)\,,\label{condfi1} 
\end{equation}
where the constants $C$ and $C_j$ are independent of $N$.

We multiply \eqref{apr} by $\varphi _\N u$ and, for $t$ fixed, integrate in $\R$. Thus, by applying integration by parts we obtain that
\begin{align}
\frac12\frac{d}{dt}\int\varphi_\N  w^2&=-\textstyle{\frac{2k+1}2}\int\varphi_\N '(\partial_x^kw)^2+c_{k-1}\int\varphi_\N ^{(3)}(\partial_x^{k-1}w)^2+\cdots+c_1\int \varphi_\N ^{(2k-1)}(\partial_x w)^2+\frac 12\int\varphi_\N ^{(2k+1)}w^2\notag\\
&\quad -\int( P(z_1)-P(z_2))\,\varphi_\N  w\,.
\label{apriori}
\end{align}
The integration by parts is justified as follows:
since there is a constant $C>0$ such that $\|(1+x_+)^{\alpha}u_i(t)\|_{L^2}\leq C$, and $\|u_i(t)\|_{H^{n+1}(\mathbb R)}\leq C$ for all $t\in[0,1]$ and $i=1,2$, by using integration by parts and  truncation functions, it can be proved that the following interpolation property holds:
\begin{equation}
\|(1+x_+)^{\alpha(1-\frac j{(n+1)})}\partial_x^j u_i(t)\|_{\ldos}\leq C\,,\quad \text{for all }t\in[0,1], \quad i=1,2\,,\quad j=0,\cdots,n+1. \label{inter}
\end{equation}
Since $\alpha>(n+1)/4$, it follows that, for $0\leq j\leq k$,  $(1+x_+)^{(k+2)/4}\partial ^j_xw(t)\in\ldos$, and thus, from \eqref{condfi2} and \eqref{condfi}  $\varphi^{(l)}\partial_x^jw(t)\in\ldos$ for all positive integers $l$. This implies that all the terms which appear in the procedure to obtain  \eqref{apriori} are in a right setting for the application of  the integration by parts.

Let us estimate the last term  on the rand-hand side of \eqref{apriori}. From \eqref{poli1} we have that
\begin{equation}
P(z)=\sum_{d=2}^{k+1}A_d(z)\quad\text{where}\quad A_d(z)=\kern-20pt\sum_{\substack{|m|=n-2(d-1)\\0\leq m_1\leq\cdots\leq m_d}} \kern-15pt a_{d,m}\,\partial_x^{m_1}u\cdots\partial_x^{m_d}u\,,\label{P}
\end{equation}
and thus 
\begin{equation}
|\int\bigl(P(z_1)-P(z_2)\bigr)\varphi_\N \,w\,dx|=|\sum_{d=2}^{k+1}\int\bigl(A_d(z_1)-A_d(z_2)\bigr)\varphi_\N \,w\,dx|\equiv|\sum_{d=2}^{k+1}I_d|\,.\label{suma1}
\end{equation}
It is easily seen that
\begin{equation}
A_d(z_1)-A_d(z_2)=\sum_{\substack{|m|=n-2(d-1)\\0\leq m_1\leq\cdots\leq m_d}} \kern-15pt  a_{d,m}\bigl(\partial_x^{m_1}w\partial_x^{m_2}u_1\cdots\partial_x^{m_d}u_1+\partial_x^{m_1}u_2\partial_x^{m_2}w\cdots\partial_x^{m_d}u_1+\partial_x^{m_1}u_2\partial_x^{m_2}u_2\cdots\partial_x^{m_d}w\bigr).\label{Ad}
\end{equation}
We estimate $I_2$, which,   having the derivatives of the highest order, is the most critical term in \eqref{suma1}. From \eqref{Ad},

\begin{equation}I_2=\int\bigl(A_2(z_1)-A_2(z_2)\bigr)\varphi_\N \,w\,dx=\sum_{\substack{m_1+m_2=n-2\\ 0\leq m_1\leq m_2}}a_{2,m}\int(\partial_x^{m_1}w\partial_x^{m_2}u_1+\partial_x^{m_1}u_2\partial_x^{m_2}w)\varphi_\N \,w\,dx.\label{Idos}\end{equation}

We estimate only the second term on the right-hand side of \eqref{Idos}, the other term being similar. We apply integration by parts to reduce the order of $\partial_x^{m_2}w$ and obtain that
\begin{equation}\int(\partial_x^{m_1}u_2)\,(\partial_x^{m_2}w)\,\varphi_\N \,w=\kern-40pt\sum_{\kern40pt\substack{\\ \\r_1+r_2+2r_3=m_1+m_2=n-2\\r_1\geq m_1}}\kern-40pt c_{r_1,r_2}\int(\partial^{r_1}_xu_2)\,\varphi_\N ^{(r_2)}\,(\partial_x^{r_3}w)^2\label{I2}\end{equation}
To analyse the terms in this sum we consider the cases $r_2=0$ and $r_2\geq 1$: 

(i) If $r_2=0$ and $r_3=0$, then  $r_1=n-2$ and  we bound the corresponding integral in \eqref{I2}  by $$C\|\partial_x^{n-2}u_2(t)\|_{L^{\infty}}\int\varphi_\N \,w^2\leq C\int\varphi_\N \,w^2,$$ 
where $C$ is independent of $t\in[0,1]$ by the Sobolev embedding of $H^1(\mathbb R)$ in $L^\infty(\mathbb R)$.

If $r_2=0$ and $r_3\geq 1$, then the maximum value of $r_1$ in \eqref{I2} is $n-4$. Therefore, using the fact that 
 $\varphi_\N \leq  C(1+x_+)\varphi'_\N $ we bound the corresponding integral in \eqref{I2} by
 \begin{equation}C\max_{0\leq r_1\leq n-4}\|(1+x_+)\partial^{r_1}_xu_2(t)\|_{L^{\infty}}\int\varphi_\N '\,(\partial_x^{r_3}w)^2.\label{apriori3}\end{equation}
 From \eqref{inter}  it can be seen that if $\Psi\in C^{\infty}(\R)$ is a truncation function with $\Psi\equiv0$ in $(-\infty,1]$, and $\Psi\equiv 1$ in $[2,+\infty)$, then $\Psi(\cdot)(1+x_+)^{\alpha(1-\frac {j+1}{n+1})}\partial_x^j u_i(t)\in H^1(\R)$, $i=1,2$,  $j=0,1,\cdots,n$, and
 \begin{equation}\|\Psi(\cdot)(1+x_+)^{ \alpha(1-\frac {j+1}{n+1})}\partial_x^j u_i(t)\|_{H^1(\mathbb R)}\leq C(\|(1+x_+)^{\alpha}u_i(t)\|_{\ldos}\,,\|u_i(t)\|_{H^{n+1}(\R)})\leq C\,\quad\text{for all }t\in[0,1]. \label{alpfa1} \end{equation}
 In particular, for $j=0,\cdots,n-4$, $\alpha(1-\frac {j+1}{n+1})> \frac{n+1}4(1-\frac{n-3}{n+1})= 1$, and thus from the Sobolev embedding of $H^1$ in $L^\infty$ we conclude that
 \begin{equation}\max_{0\leq j\leq n-4}\|(1+x_+)\partial_x^ju_i(t)\|_{L^\infty(\R)} \leq C,\label{Idos1}\end{equation}
 with $C$ independent of $t\in[0,1]$. Thus \eqref{apriori3} is bounded by $C\int\varphi_\N '(\partial_x^{r_3}w)^2$.

 (ii) If $r_2\geq1$, then $r_1\leq n-3$. From \eqref{condfi} $\varphi^{(r_2)}\leq C_{r_2}\,\varphi' \leq C\varphi'$ for $1\leq r_2\leq n-2$. Thus we bound the corresponding term in \eqref{I2} by 
 \begin{equation}C\max_{0\leq r_1\leq n-3}\|\partial_x^{r_1}u_2(t)\|_{L^{\infty}}\int\varphi_\N '\,(\partial_x^{r_3}w)^2\leq C \int\varphi_\N '\,(\partial_x^{r_3}w)^2.\label{Idos3}\end{equation}
 
 Now, let us determine the maximum value of $r_3$ appearing in \eqref{I2}. For that, we observe that  $r_1+r_2+2r_3=n-2$ is odd, and thus the maximum value of $r_3$ occurs when $(r_1,r_2)=(0,1)$ or $(1,0)$, which then gives $r_3\leq (n-3)/2=k-1$. 
 
 Gathering the estimates of the cases (i) and (ii) above, and taking into account that $r_3\leq k-1$, we conclude that
 \begin{equation}|I_2|\leq C\sum_{j=1}^{k-1}\int\varphi_\N '(\partial_x^jw)^2\,dx +C\int\varphi_\N  w^2\,.\end{equation}
 Proceeding in a similar way, we obtain the same bound for $|I_3|,\cdots |I_{k+1}|$, and thus  for the left hand side of \eqref{suma1}. 
 Therefore, returning to \eqref{apriori} and using the fact that, from condition \eqref{condfi}, $|\phi^{(j)}|\leq C\,\phi'$, $j=1,\cdots 2k+1$, we obtain

 \begin{equation}
\frac12\frac{d}{dt}\int\varphi_\N  w^2\leq-{\textstyle{\frac{2k+1}2}}\int\varphi_\N '(\partial_x^kw)^2 + C\sum_{j=1}^{k-1}\int\varphi_\N '(\partial_x^jw)^2\,dx +C\int\varphi_\N  w^2\label{apriori2}
\end{equation}

To handle the terms in \eqref{apriori2} having derivatives $\partial_x^jw$ with $j=1,\cdots, k-1$, we will show that given $\varepsilon>0$ there is a constant $C_\varepsilon>0$ such that 
for $j=1,\cdots,k-1$ 
\begin{equation} \int\varphi_\N '( \partial_x^{j}w)^2\leq \varepsilon\int\varphi_\N '(\partial^k_xw)^2+C_\varepsilon\int\varphi_\N  w^2\,.\label{apriori0}\end{equation}
In fact, we first prove that 
\begin{equation}
\int\varphi_\N '(\partial_x^jw)^2\leq\varepsilon\int\varphi_\N '(\partial_x^{j+1}w)^2+C_\varepsilon\int\varphi_\N  w^2\,.\label{apriori1}\end{equation}
This can be seen by induction: by applying integration by parts,  Young's inequality $|xy|\leq  \frac1{2\varepsilon} x^2+\frac{\varepsilon}{2}y^2$, and the properties of $\varphi_N$ we can see that \eqref{apriori1} is valid for $j=1$. If we assume that \eqref{apriori1} is valid for $j-1$, then, again  integrating by  parts and  using Young's inequality
\begin{align*}
\int\varphi_\N '(\partial_x^jw)^2&=\frac12\int\varphi_\N ^{(3)}(\partial_x^{j-1}w)^2-\int\varphi_\N '\partial_x^{j-1}w\,\,\partial_x^{j+1}w\\
&\leq C\int\varphi_\N '(\partial_x^{j-1}w)^2+\frac 1{2\varepsilon}\int \varphi_\N '(\partial_x^{j-1}w)^2+\frac\varepsilon 2\int\varphi_\N '(\partial_x^{j+1}w)^2.
\end{align*}
If we apply the induction hypothesis at level $j-1$, with $\frac{1/2}{C+1/{(2\varepsilon)}}$ instead of $\varepsilon$, then we have that
\begin{equation}
\int\varphi_\N '(\partial_x^jw)^2\leq \textstyle{\frac12}\int\varphi_\N '(\partial_x^jw)^2+C_\varepsilon\int \varphi_\N  w^2+\frac\varepsilon2\int\varphi_\N '(\partial_x^{j+1}w)^2\,,
\end{equation}
which gives \eqref{apriori1}. From a repeated application of  \eqref{apriori1} we obtain \eqref{apriori0}.

Therefore, taking into account that the first term on the right-hand side of  \eqref{apriori2}  is negative,  we can apply \eqref{apriori0}  with $\varepsilon$ sufficiently small,  to absorb with this negative term the integrals containing $(\partial_x^jw)^2$ in \eqref{apriori2}. Thus we conclude that 
\begin{equation}\frac{d}{dt}\int\varphi_\N  w^ 2\leq C_\varepsilon\int\varphi_\N  w^2,\label{tres1}\end{equation}
which, from Gronwall's inequality and \eqref{condfi3} implies that 
\begin{equation*}\int\varphi_\N  w(t)^2\,dx\leq C\int\varphi_\N  w(0)^2\,dx\leq C\int e^{\beta x}w(0)^2dx\,,\quad\text{for all } t\in[0,1] ,
\end{equation*} 
where $C$ is independent of $t\in[0,1]$ and $N$.  Since $\varphi_\N (x)\to e^{\beta x}$ as $N\to\infty$, the conclusion of the theorem will follow by applying Fatou's Lemma on the left-hand side of the former inequality.

 In this way, the proof of theorem \ref{decaimiento} will be complete if we construct a sequence of functions $\varphi_\N $, satisfying the conditions \eqref{condfi3} to \eqref{condfi1}. For that we proceed as follows: 
 
 Let $\tilde\phi \in C^\infty(\R)$ be a nonincreasing function such that $\tilde\phi(x)=1$ for $x\in (-\infty, 0]$, and $\tilde\phi(x)=0$ for $x\in[1,\infty)$. For each $N\in\mathbb N$ let $\phi_{\N}(x)\equiv\phi(x) :=\tilde\phi(x-N)$. Thus $\phi$ is supported in $(-\infty, N+1]$, and $(1-\phi)$ in   $[N,+\infty)$. We define 
 \begin{equation}\theta_\N (x)\equiv\theta(x):=\phi \beta e^{\beta x}+(1-\phi)\beta e^{\beta N}\label{tetan}\end{equation}
 and 
 \begin{equation}\varphi_\N (x)\equiv\varphi(x):=\int_{-\infty}^x\theta(x')\,dx'\,.\label{fifi}\end{equation}
 Let us show that $\varphi_\N $ satisfies the conditons \eqref{condfi3} to \eqref{condfi1}.
 
Taking into account the support of $(1-\phi)$, we see that $0\leq \theta(x)\leq \phi\beta e^{\beta x}+(1-\phi)\beta e^{\beta x}=
\beta e^{\beta x }$. Thus, by integrating $\varphi'$ we have that $0\leq\varphi(x)\leq e^{\beta x}$. Besides, from the definition of $\varphi$ it is clear that $\varphi_\N (x)\to e^{\beta x}$ as $N\to\infty$. Thus $\varphi$ satisfies \eqref{condfi3}.

To prove \eqref{condfi2} it suffices to observe that for $x\leq N$, $\varphi(x)\leq e^{\beta N}\leq C_{N}(1+x_+)^{(k+2)/4}$, while for $x>N$, 
\begin{equation}\varphi(x)\leq \int_{-\infty}^{N+1}\beta e^{\beta x'} dx'+\int_N ^x\beta e^{\beta N} dx'\leq e^{\beta(N+1)}+ x\beta e^{\beta N}\leq C_{N}(1+x_+)^{(k+2)/4},\label{es}\end{equation} since $k\geq 2$. Thus we have \eqref{condfi2}.

We proceed now to prove \eqref{condfi1}. For $x\leq N$, $\varphi(x)=e^{\beta x}=\frac1\beta\varphi'(x)\leq C (1+x_+)\varphi'(x)$. If $x> N$, then, from \eqref{es}, and using the fact that $x\geq 1$, we see that
\begin{equation}
\varphi(x)\leq  e^{\beta(N+1)}+ x\beta e^{\beta N}\leq (\frac1{\beta}+1)e^\beta x \beta  e^{\beta N}\,.\label{fiprima2}
\end{equation}
On the other hand, for $x>N$, 
\begin{equation} x\varphi'(x)=x\theta(x)\geq N\phi\beta e^{\beta N}+x(1-\phi)\beta e^{\beta N}.\label{fiprima}\end{equation}
Therefore, from \eqref{fiprima2} and \eqref{fiprima}, taking into account the supports of $\phi$ and $(1-\phi)$ we observe   that  for $x>N+1$, $ x\varphi'(x)\geq C\varphi(x)$, while for $N<x<N+1$,  we conclude  that
\begin{equation*}x\varphi'(x)\geq N\phi\beta e^{\beta N}+N(1-\phi)\beta e^{\beta N}=N\beta e^{\beta N}\geq  {\textstyle \frac12}(N+1)
\beta e^{\beta N}\notag \geq {\textstyle \frac12}x\beta e^{\beta N}\geq C\varphi(x),\label{thetaphi}\end{equation*}
from which \eqref{condfi1} follows.

Finally, we verify \eqref{condfi}. We observe that that for $j=1,2\cdots,$ and fixed $\beta>0$, 
 \begin{align*}
 |\vp^{(j+1)}|&=|\theta^{(j)}|=|\phi \beta^{1+j}e^{\beta x}+\sum_{l=1}^j c_{j,l}\phi^{(l)}\beta^{j-l}\beta e^{\beta x}- \phi^{(j)}\beta e^{\beta N}|\\
 &\leq \beta^j\phi\beta e^{\beta x}+C_j(1+\beta)^{j-1}(\beta e^{\beta(N+1)}+\beta e^{\beta N})\chi_{[N,N+1]}\\
 &\leq \beta^j\theta+C_j\beta e^{\beta N}\chi_{[N,N+1]}= \beta^{j}\theta+C_j(\beta\phi e^{\beta N}+(1-\phi)\beta e^{\beta N})\chi_{[N,N+1]}\\
 &\leq \beta^j\theta +C_j\theta =C_{j}\varphi',
 \end{align*}
 where $C_j$ depends on $\beta$ and $j$ but is independent of $N$. Thus, the first inequality in \eqref{condfi} is proved. For the  inequality $\phi' \leq C\,\phi$   in \eqref{condfi} we proceed by integrating the inequality $\phi''\leq C\,\phi'$ already established. This completes the proof of Theorem \ref{decaimiento}.
 \end{proof}
\begin{remark}\label{remdec} {\rm For the case of equation \eqref{ec}, with $p\leq k$, we can establish a result similar to Theorem \ref{decaimiento},  by making minor modifications and some simplifications in the former proof. In the simple case $p\leq 1$, for example for the equation
$$\partial_tu+(-1)^{k+1}\partial_x^nu=-u\partial_xu,$$
it is possible to follow the procedure of the proof of Theorem \ref{decaimiento} to establish, without the hypothesis of polynomial decay, that the exponential decay at $t=0$ is preserved for $t\in[0,1]$.  This can be done by taking $\varphi_\N(x):=\int_{-\infty}^x\theta_\N(x')\,dx'$ as in \eqref{fifi}, with $\theta_\N(x)\equiv\theta(x):=\phi \beta e^{\beta x}+(1-\phi)\beta e^{-\beta(x-2N) }$, instead of the functions $\theta_N$ defined in \eqref{tetan}. This functions $\varphi_\N$ are bounded and satisfy \eqref{condfi3} and \eqref{condfi} which is enough for this case.
}
\end{remark}
\section{Estimates of Carleman Type}\label{sectioncar}
In this section we obtain boundedness properties of the linear operator $(\partial_t+(-1)^{k+1}\partial_x^n)^{-1}$, and its spatial derivatives up to order $n-1$, in spaces of the type $L^p-L^q$ with exponential weight $e^{\lambda x}$. We  keep our exposition simple since we only use  values of  $p$ and $q$ in the set $\{1,2,+\infty\}$.

Let $D:=\R\times[0,1]$ and, for $R\in \mathbb R$, let $D_R:=\{(x,t)\mid x\geq R\, ,\,t\in[0,1]\}$. We will denote
\[\|\;\|_{L^p_xL^q_T}:=\|\;\|_{L^p_xL^q_t(D)},\;\;\|\;\|_{L^p_{x\geq R}L^q_T}:=\|\;\|_{L^p_xL^q_t(D_R)}, \;\;\|\;\|_{L^q_TL^p_x}:=\|\;\|_{L^q_tL^p_x(D)},\;\; \text{and }\;\;\|\;\|_{L^q_TL^p_{x\geq R}}:=\|\;\|_{L^q_tL^p_x(D_R)}.\]
\begin{Theorem}\label{car} For $k\in\mathbb N$ and $n=2k+1$, let $v\in C([0,1];H^{n}(\mathbb R))\cap C^1([0,1];\ldos)$ be a function such that $supp\,v(t)\subset[-M,M]$ for all $t\in[0,1]$, for some $M>0$. Then, for $\lambda>2$ we have that
\begin{equation}
\|\el v\|\liTldx\leq C\|\el(|v(0)|+|v(1)|)\|_\ldos +C\|\el(\partial_t+(-1)^{k+1}\partial_x^{n})v\|\luTldx.\label{car1}
\end{equation}
\begin{equation}
\sum_{j=1}^{n-1}\|\el\partial_x^jv\|\lixldT\leq C\lambda^{n-1}\|( |J^{n} (\el v(1))| +|J^{n}(\el v(0))|\|_\ldos+\|\el(\partial_t+(-1)^{k+1}\partial_x^{n})v\|\luxldT,\label{car2}
\end{equation}
where $C$ is independent of $\lambda>2$ and $M$, and $(Jf)\,\widehat{\;}:= (1+|\xi|^2)^{1/2}\widehat f$.
\end{Theorem}
Reasoning formally, supposse that $\el(\partial_t+(-1)^{k+1}\partial^n_x)g=h$, and denote $f=\el g$ and 
$T_0=[\el (\partial_t+(-1)^{k+1}\partial^n_x)e^{-\lambda x}]^{-1}$. 
Then,  $f=T_0h$. Since $\el\partial_xe^{-\lambda x}f=(\partial_x-\lambda)f$, we have that $\el\partial_x^ne^{-\lambda x}f=(\partial_x-\lambda)^nf$, and thus the multipier operator representing $T_0$ via the Fourier transform is given by
\begin{equation}
(T_0h)\,\widehat{\;}\,(\xi,\tau)=\frac{\widehat h}{i\tau+(-1)^{k+1}(i\xi-\lambda)^n}\equiv m_0\widehat h.\end{equation}
We will write $m_0$ as 
\begin{equation}
m_0=\frac{-i}{\tau-(\xi+i\lambda)^n}.\label{tres}
\end{equation}
Since for a positive integer $j$, $\el\partial^j_xg=(\partial_x-\lambda)^jf$, we have that 
\begin{equation}\el\partial^j_xg=(\partial_x-\lambda)^jT_0h\equiv T_jh=[(i\xi-\lambda)^jm_0\widehat h]\,\widecheck{\;}=\Bigl[\frac{-i^{j+1}(\xi+i\lambda)^j\,}{\tau-(\xi+i\lambda)^n}\widehat h\Bigr]\,{}^{\textstyle{\widecheck{\;}}}\,\equiv[m_j\widehat h]\,\widecheck{\;}. \label{cuatro}\end{equation}

\begin{lemma}\label{lema1}
Let $h\in L^1(\R^2)$. Then there is $C>0$ independent of $h$ and $\lambda>0$ such that \begin{equation}
\|T_0h\|\litldx\leq C\|h\|\lutldx.
\label{cinco}\end{equation} 

\end{lemma}
\begin{proof}
Let $a(\xi)$ and $b(\xi)$ be the real and imaginary parts of $-(\xi+i\lambda)^n$, respectively. Then
\[m_0=\frac{-i}{\tau+a(\xi)+ib(\xi)}.\]
We recall that for $a\in \R$ and $b\not=0$
\begin{equation}\bigl(\frac{1}{\tau+a+ib}\bigr)\,\widecheck{\;}^{{}_\tau}\,(t)=c\,e^{-iat}[e^{-bt}\chi_{[0,\infty]}(t)\chi_{(0,\infty)}(b)+e^{bt}\chi_{[-\infty,0)}(t)\chi_{(-\infty, 0)}(b)]=:G_{a,b}(t).\label{siete}\end{equation}
 Since $|G_{a,b}|\leq c$, by taking inverse Fourier transform in the varible $\tau$ and using convolutions, it follows that for $t\in\R$
\[
|[T_0h(t)]\,\widehat{\;}\,(\xi)|=|\int_{-\infty}^\infty G_{a(\xi),b(\xi)} (t-s)\widehat{h(s)}(\xi)\,ds|\leq C\int_{-\infty}^{\infty}|\widehat{h(s)}(\xi)|\,ds,\]
for those values of $\xi$ such that $b(\xi)\not=0$ (a finite set). In this way, applying Plancherel's identity and Minkowski's integral inequality we obtain \eqref{cinco}.
\end{proof}
\begin{lemma}\label{lema2} 
Let $h\in L^1(\R^2)$. Then there is $C>0$, independent of $h$ and $\lambda>2$, such that 
\begin{equation}
\|T_jh\|\lixldt\leq C\|h\|\luxldt\quad\text{for }  j=1,\cdots,n-1.\label{seis}\end{equation}
 
\end{lemma}
\begin{proof}
From \eqref{tres} and \eqref{cuatro} 
\begin{equation}
(T_jh)\,\widehat{\;}\,=(i\xi-\lambda)^jm_0\widehat h=\frac{-i^{j+1}(\xi+i\lambda)^j\,\widehat h}{\tau-(\xi+i\lambda)^n}=m_j\widehat h.\label{ocho}\end{equation}
Let us denote $\theta:=(\xi+i\lambda)/\tau^{1/n}$. Then
\begin{equation}
m_j=\frac{C}{\tau^{1-j/n}}\frac{\theta^j}{1-\theta^n}=\frac{C}{\tau^{1-j/n}}\sum_{l=1}^n\frac{c_l}{\theta-r_l},
\end{equation}
where $r_l:=a_l+ib_l$, $l=1,\cdots,n$, are the $n^{\text{th}}$-roots of $ 1$, and the $c_l$ can be computed by L'Hopital's rule to obtain that
\[c_l=\lim_{\theta\to r_l}\frac{(\theta-r_l)\theta^j}{1-\theta^n}=-\frac{1}{nr_l^{n-j-1}}.\]
Therefore
\[m_j=\frac{C}{\tau^{1-(j+1)/n}}\sum_{l=1}^n\frac{c_l}{\xi+i\lambda-\tau^{1/n}r_l}=\frac{C}{\tau^{1-(j+1)/n}}\sum_{l=1}^n\frac{c_l}{\xi-\tau^{1/n}a_l+i(\lambda-\tau^{1/n}b_l)}.
\]
Taking the inverse Fourier transform with respect to the variable $\xi$, observing that for fixed $\lambda$, $\lambda-\tau^{1/m}b_l\not=0$ for all $l$,  except for a finite number of values of $\tau$, and applying \eqref{siete} (with $\xi$ and $x$ instead of $\tau$ and $t$, respectively) we obtain a collection of bounded functions $G_1,\cdots,G_l$ of $x$ and $\tau$ such that
\begin{equation*} [m_j(\cdot_{\xi},\tau)]\,\widecheck{\;}^{{}_{\,\xi}}\,(x)=\frac{C}{\tau^{1-(j+1)/n}}\sum_{l=1}^n c_l\,G_l(x,\tau).\end{equation*}
If $|\tau|>1$, then it is clear that 
\begin{equation}
	|[m_j(\cdot,\tau)]\,\widecheck{\:}^{{}_{\,\xi}}\,(x)|\leq C,\label{inve}
\end{equation}
with $C$ independent of $\lambda$, $x$ and $|\tau|>1$.
 We can use \eqref{ocho} to prove that this function is bounded also for $|\tau|\leq 1$. To do this we will consider only the case $j=n-1=2k$, the other cases being similar. Let us observe from \eqref{cuatro} that 
 \begin{equation*}\bigl |m_{2k}-\frac{i^{n}}{\xi+i\lambda}\bigr |=\bigl|\frac{(\xi+i\lambda)^{n-1}}{\tau-(\xi+i\lambda)^n}+\frac1{\xi+i\lambda}\bigr|=\bigl|\frac\tau{(\tau-(\xi+i\lambda)^n)(\xi+i\lambda)}\bigr|\leq \frac2{|\xi+i\lambda|^{n+1}}\in L_\xi^1,\end{equation*}
 since $\lambda>2$ and $|\tau|\leq 1$. From the Fourier inversion formula it can be seen that $|[|\xi+i\lambda|^{-n-1}]\,\widecheck{\;}(x)|\leq C\,$, with $C$ independent of $\lambda>2$. Thus, by taking inverse Fourier transform with respect to the variable $\xi$  and taking into account that from \eqref{siete} $[(\xi+i\lambda)^{-1}]\,\widecheck{\;}^{{}_{\,\xi}}\,(x)$ is a bounded function of $x$, with bound independent of $\lambda$, we see, together with the estimate already obtained for $|\tau|\leq 1$,  that \eqref{inve} is  valid for all $x$ and all but a finte number of values of $\tau$.
 
 Hence, we can apply basic properties of convolution and Plancherel's identity to conclude that 
 \[\|T_{2k}h\|\lixldt\leq C\|h\|\luxldt,\]
 which concludes the proof of Lemma \ref{lema2}.
 
\end{proof}

We now proceed to prove Theorem \ref{car}.

\textit{ Proof of Theorem \ref{car}:}

We extend $v$ to all $t\in\R$ with value zero in $\R-[0,1]$, and call this extension again $v$. For $\varepsilon>0$, we consider a function $\eta:=\eta_\e\in C^\infty(\R_t)$ such that $\eta_\e=0$ in $\R-[0,1]$, $\eta_\e=1$ in $[\e,1-\e]$, $\eta'\geq0$ in $[0,\e]$,   $\eta'\leq0$ in $[1-\e,1]$, and $\eta_{\epsilon'}\leq \eta_\e$ if $\e'<\e$. Define $g:=\eta_\e(\cdot_t)v$. Then 

\begin{equation*}
\el(\partial_t+(-1)^{k+1}\partial_x^n)g=\el\eta_\e'v+\eta_\e\el(\partial_t+(-1)^{k+1}\partial_x^n)v\equiv h_{1,\e}+h_{2,\e}\equiv h_1+h_2.\end{equation*}
Then,  from \eqref{cuatro}, 
\begin{equation}\el\partial^j_xg=T_jh_1+T_jh_2\quad \text{ for }j=0,\cdots,n-1.\label{docea}\end{equation}
From Lemma \ref{lema2},
\begin{equation}
\|T_jh_2\|\lixldt\leq C\|h_2\|\luxldt=C\|\eta_\e\el(\partial_t+(-1)^{k+1}\partial_x^n)v\|\luxldt.\label{doce}
\end{equation}
For $T_jh_1$,  we see from \eqref{cuatro} that  $(T_jh_1)\,\widehat{\;}=C(\xi+i\lambda)^j(T_0h_1)\,\widehat{\;}\,$, and apply  the Sobolev embedding from $H^1(\R)$ to $L^\infty(\R)$, Plancherel's identity, and Lema \ref{lema1} to conclude that
\begin{align}
\|T_jh_1\|\lixldT &\leq C\|JT_jh_1\|\ldTldx\leq C\|(1+|\xi|)(|\xi|^j+ \lambda^j)(T_0h_1)\,\widehat{\;}{\,}^x\|{{}_{L^2_T L^2_\xi}}\notag\\
&\leq C\|(1+\lambda)^j(1+|\xi|)^{j+1}(T_0h_1)\,\widehat{\;}{\,}^x\|{{}_{L^2_T L^2_\xi}} \leq C(1+\lambda)^j\|T_0J^{j+1}h_1\|{{}_{L^\infty_t L^2_x}}\notag\\
&\leq C\lambda^j\|J^{j+1}h_1\|\lutldx=C\lambda^j\|\eta_\e' J^{j+1}(\el v)\|\lutldx\leq C\lambda^{n-1}\|\eta_\e' J^{n}(\el v)\|\lutldx.\label{trece}
\end{align} 
Hence, from \eqref{docea}, \eqref{doce}, and \eqref{trece} we have that 
\begin{equation*}\|\eta_\e \el\partial^j_x v  \|\lixldT\leq C\lambda^{n-1}\|\eta_\e' J^{n}(\el v)\|\lutldx+ C\|\eta_\e\el(\partial_t+(-1)^{k+1}\partial_x^n)v\|\luxldt.\end{equation*}
We now make $\e\to 0$ in this inequality and apply Fatou's Lemma on the left-hand side and the monotone convergence theorem in the second term of the right-hand side.  For the first term on the right-hand side we use the fact that $|\eta'_{\e}|$ acts as an  approximation of the  identity  on each one of the time intervals $(0,\e)$ and $(1-\e,1)$. Thus,  we obtain  \eqref{car2} after adding up in $j$. 

The proof of \eqref{car1} is similar but we use Lemma \ref{lema1} instead of Lemma \ref{lema2}. This completes the proof of Theorem \ref{car}.\qed
\section{Lower estimates}\label{sectionlow}
In this section we prove that the $L^2$-norm of $w$ in a small rectangle $Q=[0,1]\times[r,1-r]$, with $r\in(0,1)$ can be bounded by a multiple of the $H^{n-1}$ norm of $w$ in a distant rectangle $[R,R+1]\times[0,1]$. 
\begin{lemma}\label{lema3}
Let $\phi\in C^\infty([0,1])$ be a function with $\phi(0)=\phi(1)=0$, and for $R>1$ and $ \alpha>0$ define $\psi_\alpha(x,t):=\psi(x,t):=\alpha(\frac xR+\phi(t))^2$, $x\in\R$,   $t\in[0,1]$. For $n=2k+1$, suppose that $g\in C([0,1];H^n(\R))\cap C^1([0,1];L^2(\R))$ is such that $g(0)=g(1)=0$ and  the support of $g$ is contained in the set
\begin{equation*}
A_{1,5}:=\{(x,t)\mid t\in[0,1]\,,\;1\leq \frac xR+\phi(t)\leq 5\}.
\end{equation*}
Then, there is a  constant $C=C(n)>0$  and a constant $\overline C=\overline C(n,\|\phi'\|_{L^\infty},\|\phi''\|_{L^\infty})>1$ such that
\begin{equation}
\sum_{j=0}^{n-1}\frac{\alpha^{n-j-\frac12}}{R^{n-j}}\|e^{\psi_\alpha}\partial_x^jg\|\leq C\|e^{\psi_\alpha}(\partial_t+(-1)^{k+1}\partial_x^n)g\|\quad\text{for all }\alpha>\overline CR^{n/(n-1)},\label{quince}
\end{equation}
where $\|\cdot\|:=\|\cdot\|_{L^2(D)}$.
\end{lemma}
\begin{proof}
Let $f:=e^\psi g$ and observe that $e^\psi\partial_x^j g=e^\psi\partial_x^je^{-\psi}f=(\partial_x-\psi_x)^jf$, and $e^\psi\partial_t g=(\partial_t-\psi_t)f$. Therefore, if we denote $$T:= e^\psi(\partial_t+(-1)^{k+1}\partial_x^n)e^{-\psi}=(\partial_t-\psi_t)+(-1)^{k+1}(\partial_x-\psi_x)^n,$$ then, to prove the inequality in \eqref{quince} we must prove that
\begin{equation}
\sum_{j=0}^{n-1}\frac{\alpha^{n-j-\frac12}}{R^{n-j}}\|(\partial_x-\psi_x)^jf\|\leq C\|Tf\|.\label{dxmb}
\end{equation}
Let $B:=-\psi_x=-\frac{2\alpha} R\varphi$, where $\varphi(x,t):=\frac xR+\phi(t).$ We will study the operator $(\partial_x-\psi_x)^n=(\partial_x+B)^n$ in the following manner:

Each one of the terms in the expansion of this operator is of the form $T_1\cdots T_n $, where each $T_i$ is either $\partial_x$ or $B$. For $m$ and $l$ with $m+l=n$ we will denote by $[m,l]$ the sum of the terms $T_1\cdots T_n$ in this  expansion with $T_i=\partial_x$ for $m$ values of $i$, and $T_i=B$ for $l$ values of $i$. The number of terms of  $[m,l]$ in the binomial expansion of  $(\partial_x+B)^n$ is then given by $\binom nl=\binom nm:=n!/m!\,l!$.
Applying integration by parts in $D$, for the class of functions satisfying the hypotheses given for $g$, we observe that $[m,l]$ is a symmetric operator if $m$ is even and is an antisymmetric operator if $m$ is odd.

In this way, we write
\begin{equation}
(-1)^{k+1}T=(-1)^{k+1}(\partial_t-\psi_t)+\sum_{\substack{0\leq m\leq n\\l+m=n}}[m,l]:=S+A,
\end{equation}
where 
\begin{equation}\begin{array}{r c c c c c c c c c c l}
S&=&[n-1,1]&+&[n-3,3]&+&\cdots&+&[2,n-2]&+&[\,0\,,\,n\,]+(-1)^k\psi_t&\equiv S_1+(-1)^k\psi_t,\\
A&=&[\,n\,,\,0\,]&+&[n-2,2]&+&\cdots&+&[3,n-3]&+&[1,n-1]-(-1)^k\partial_t&\equiv A_1-(-1)^k\partial_t\label{SAdos}.
\end{array}
\end{equation}
Let us denote by $\langle\cdot,\cdot\rangle$ the inner product in the (real) space $L^2(D)$. Then
\begin{equation}
\|Tf\|^2= \langle (S+A)f,(S+A)f\rangle=\|Sf\|^2+\|Af\|^2+2\langle Sf,Af\rangle \geq 2\langle Sf,Af \rangle.\label{Te}\end{equation}
Now,
\begin{equation}
\langle Sf  ,  Af \rangle=\langle S_1f   ,  A_1f \rangle- (-1)^k\langle S_1f   ,  \partial_tf \rangle+(-1)^k\langle \psi_t f  ,  A_1f \rangle-\langle \psi_tf  , \partial_tf  \rangle.\label{suno}\end{equation}
We will now estimate each one of the four terms on the right hand side of \eqref{suno}.

To estimate $\langle S_1f,A_1f\rangle$ we observe that this product is a sum of terms of the form $\langle [m,n-m]f,[r,n-r]f\rangle$, with $m$ even and $r$ odd, say $m=2k_1$, $r=2k_2+1$, $k_1, k_2\in\{0,\cdots,k\}$. Using the fact that $B_{xx}=\psi_{xxx}=0$, we can apply integration by parts to obtain that
\begin{equation}
\langle [m,n-m]f,[r,n-r]f\rangle=\sum_{j=0}^{k_1+k_2}\int_DP_{k_1,k_2,j}(B,B_x)(\partial_x^jf)^2
 \end{equation}
where  $P_{k_1,k_2,j}(B,B_x)=c_{k_1,k_2,j}B^{(n-m)+(n-r)-\nu}B_x^\nu$, with $\nu+2j=m+r$.
Since $B=-\psi_x=-\frac{2\alpha}R\varphi$, $B_x=-\psi_{xx} =-\frac{2\alpha}{R^2}$, we have that
\begin{align} P_{k_1,k_2,j}(B,B_x)&=-c_{k_1,k_2,j}(2\alpha)^{2n-m-r}
\frac{\varphi^{2n-m-r-\nu}}{R^{2n-m-r+\nu}} \notag\\
&=-(2\alpha)^{2n-2k_1-2k_2-1}c_{k_1,k_2,j}\frac{\varphi^{2n-4k_1-4k_2+2j-2}}{R^{2n-2j}} \equiv \alpha^{2n-2k_1-2k_2-1}Q_{k_1,k_2,j}\label{p}
\end{align} 
 Therefore, 
 \begin{align} 
 \langle S_1f,A_1f\rangle&=\sum_{k_1,k_2=0}^k\langle[2k_1,n-2k_1]f,[2k_2+1,n-2k_2-1]f\rangle\notag\\
 &=\sum_{k_1,k_2=0}^k\alpha^{2n-2k_1-2k_2-1}\sum_{j=0}^{k_1+k_2}\int Q_{k_1,k_2,j}(\partial_x^jf)^2\notag\\
 &=\sum_{j=0}^{2k}\sum_{\substack{k_1+k_2\geq j\\0\leq k_1,k_2\leq k}}\alpha^{2n-2k_1-2k_2-1}\int Q_{k_1,k_2,j}(\partial_x^jf)^2.\label{mayorj}
 \end{align}
 For each $j$ in the former expression we will separate the term having the highest power of $\alpha$. Thus we write,
 \begin{equation}
 \langle S_1f,A_1f\rangle=\kern-2pt\sum_{j=0}^{2k}\kern-1pt\alpha^{2n-2j-1}\kern-15pt\sum_{\substack{k_1+k_2=j\\0\leq k_1,k_2\leq k}}\kern-5pt\int Q_{k_1,k_2,j}(\partial_x^jf)^2+\sum_{j=0}^{2k}\sum_{\substack{k_1+k_2> j\\0\leq k_1,k_2\leq k}}\kern-5pt\kern-6pt\alpha^{2n-2k_1-2k_2-1}\int Q_{k_1,k_2,j}(\partial_x^jf)^2\equiv\sum_{j=0}^{2k} I_j+\sum_{j=0}^{2k} II_j.\label{inf3}\end{equation}
We will concentrate upon the terms $I_j$ and will refer to the terms $II_j$ as lower order terms (l.o.t.).

For each $j$ we will now compute the term $I_j$. If  $m\leq n$ and $l=n-m$, then $[m,l]$ is a sum of operators of the form $T=T_1T_2\cdots T_n$  where $T_i=\partial_x$ for $m$ indices $i$, and $T_i=B$ for the remaining $l$ indices $i$.  We apply the product rule for derivatives to expand $T$ and consider the terms  in its expansion  containing the derivatives of highest order: $\partial^m_x$ and $\partial^{m-1}_x$. This leads to (see the illustration below)
\begin{equation}T=T_1\cdots T_n=B^l\partial_x^m+[B^{r_1}(\partial_xB^{l-r_1})+B^{r_2}(\partial_xB^{l-r_2})+\cdots+B^{r_m}(\partial_xB^{l-r_m})]\partial_x^{m-1}+\text{l.d.t.},\label{T}\end{equation}
where $r_1,r_2,\cdots r_n$ depend upon the position of the $m$ operators $\partial_x$ in the expression of $T$, and the notation l.d.t. stands for ``lower derivative terms".

To illustrate this, let us take for example the case with $n=9$, $m=3$, $l=6$, and consider the operator 
\begin{align*}T&=B\partial_xB\partial_xBB\partial_xBB. \;\text{ Then,}\\
T&=B^6\partial_x^3+B^4(\partial_xB^2)\partial_x^2+B^2(\partial_xB^4)\partial_x^2+B(\partial_xB^5)\partial_x^2+\text{l.d.t.}
\end{align*}
But, the operator  $T^{(*)}:=T_n\cdots T_2 T_1$,  is also present in the expansion of $[m,l]$, and 
\begin{equation} T^{(*)}=T_n\cdots T_1=B^l\partial_x^m+[B^{l-r_1}(\partial_xB^{r_1})+B^{l-r_2}(\partial_xB^{r_2})+\cdots+B^{l-r_m}(\partial_xB^{r_m})]\partial_x^{m-1}+\text{l.d.t.}\label{T*}\end{equation} 
 Therefore, if $T\not=T^*$ we have from \eqref{T} and \eqref{T*}  that
 \begin{equation} T+T^{(*)}=2B^l\partial_x^m+m(\partial_xB^l)\partial_x^{m-1}+\text{l.d.t.}\label{TT}\end{equation} 
 It can be seen that if $T=T^{(*)}$, then  the same expression is valid.
 Since there are ${\binom nm}$ terms in the expansion of $[m,l]$ we conclude from \eqref{TT} that
 \begin{equation}
 [m,l]={\binom nm}[B^l\partial^m_x+\frac {ml}2B^{l-1}B_x\partial_x^{m-1}]+\text{l.d.t.}\label{ml}\end{equation}
 In this way, for $m=2k_1$, $r=2k_2+1$, $j=k_1+k_2=\frac12(m+r-1)$, $l=n-m$, $s=n-r$, we apply integration by parts to observe that  \begin{align*}
 \langle[m,l]f,[r,s]f \rangle&={\textstyle{{\binom nm}{\binom nr}} }\Bigl(\int B^{l+s}\partial_x^mf\,\partial_x^rf+\frac{sr}2 \int B^{l+s-1}B_x\,\partial_x^mf\,\partial_x^{r-1}f
 +\frac{ml}2 \int B^{l+s-1}B_x\,\partial_x^{m-1}f\partial_x^{r}f\Bigr) +\text{l.o.t.}\\
 &=\frac12{\textstyle{{\binom nm}{\binom nr}}} (-1)^{\frac12(m+1-r)}\Bigl((m-r)(l+s)+{rs}-{ml}\Bigr)\int B^{l+s-1}B_x(\partial_x^{\frac12(m+r-1)}f)^2+\text{l.o.t.}\\
 &=\frac12{\textstyle{{\binom nm}{\binom nr}}}(-1)^{k_1-k_2}n(m-r)\int B^{l+s-1}B_x(\partial_x^{j}f)^2+\text{l.o.t.}\\
 &=(2\alpha)^{2n-2j-1}{\textstyle{\frac12}}{\textstyle{{\binom nm}{\binom nr}}}(-1)^{j}n(m-r)\int( \frac\varphi R)^{2n-2j-2}\,\frac{-1}{R^2}\,(\partial_x^{j}f)^2+\text{l.o.t.}
 \end{align*}
 According to the definition of $I_j$ in \eqref{inf3}, and from the first equality in \eqref{mayorj},  to obtain $I_j$ we add the high order terms terms in the former expression with $k_1+k_2=j$  and $0\leq k_1,k_2\leq k$. In this way, we obtain that
 \begin{equation}
 I_j= {A^n_j\,{\frac {n}2}}\,\frac{(2\alpha)^{2n-2j-1}}{R^{2n-2j}}\int \varphi^{2n-2j-2}(\partial_x^{j}f)^2,\label{high}
 \end{equation}
 Where 
 \begin{equation}A^n_j:=(-1)^{j}\sum_{\substack{k_1+k_2=j\\ 0\leq k_1,k_2\leq k}}
 (2k_2+1-2k_1){\binom n{2k_1}}{\binom n{2k_2+1}}.\label{Asubnj2}
 \end{equation}
  
 We will prove in Lemma \ref{lemaaj} below that
  \begin{equation}A_j^n=n{\binom{n-1}j}\geq n,\quad \text{for all integers } n\geq 3\;\text{odd and all }j\leq n-1,\label{comb1}\end{equation}
  and in particular all the coefficients $A_j^n$ are positive.
  
 Regarding the lower terms $II_j$, we see from \eqref{inf3} and \eqref{p} that 
 \begin{equation}II_j=-\sum_{\substack{k_1+k_2> j\\0\leq k_1,k_2\leq k}}c_{k_1,k_2,j} \frac{(2\alpha)^{2n-2(k_1+k_2)-1}}{R^{2n-2j}}\int \varphi^{2n-4k_1-4k_2+2j-2}(\partial^j_xf)^2,\label{low}
 \end{equation}
 Thus, from \eqref{inf3}, using \eqref{high} and \eqref{low} we have that
 \begin{align}
 \langle S_1f,A_1f\rangle&=\sum_{j=0}^{2k}{A^n_j\,{\frac {n}2}}\,\frac{(2\alpha)^{2n-2j-1}}{R^{2n-2j}}\int \varphi^{2n-2j-2}(\partial_x^{j}f)^2\notag
 \\&-\sum_{j=0}^{2k}\sum_{\substack{k_1+k_2> j\\0\leq k_1,k_2\leq k}}c_{k_1,k_2,j} \frac{(2\alpha)^{2n-2(k_1+k_2)-1}}{R^{2n-2j}}\int \varphi^{2n-4k_1-4k_2+2j-2}(\partial^j_xf)^2.\label{IjmasIIj}
 \end{align}
 We now turn our attention to the product $\langle \psi_tf,A_1f\rangle$ in \eqref{suno}.  For $r=2k_1+1$, $k_1=0,\cdots,k$ and $s=n-r$, taking into account that $\psi_{txx=0}$, and using integration by parts we see that
 \begin{equation}
 \langle [r,s]f,\psi_tf]\rangle=\sum_{j=0}^{k_1}\int P_{k_1,j}(B,B_x,\psi_t,\psi_{tx})(\partial^j_xf)^2,\label{ers}
 \end{equation}
 where 
 \begin{equation}P_{k_1,j}(B,B_x,\psi_t,\psi_{tx})=c'_{k_1,j}B^{s-\nu}B_x^{\nu}\psi_t+c''_{k_1,j}B^{s-(\nu-1)}B_x^{\nu-1}\psi_{tx}\quad \text{with }\nu+2j=r.\notag\end{equation}
 Since $\psi_t=2\alpha(\frac xR+\phi)\phi'=2\alpha\varphi\phi'$ and $\psi_{tx}=2\alpha\frac{\phi'}R$, we have that
 \begin{align*}P_{k_1,j}(B,B_x,\psi_t,B_t)&=(2\alpha)^{s+1}\Bigl(c'_{k_1,j}\frac{\varphi^{s-\nu}}{R^{s-\nu}}\frac1{R^{2\nu}}\,\varphi\phi'+  c''_{k_1,j}\frac{\varphi^{s-\nu+1}}{R^{s-\nu+1}}\frac1{R^{2\nu-2}}\frac{\phi'}R\Bigr)\\
 &\equiv c_{k_1,j}\,\alpha^{s+1}\frac{\varphi^{s-\nu+1}}{R^{s+\nu}}\phi'= \alpha^{n-2k_1}\,c_{k_1,j}\frac{\varphi^{n-4k_1+2j-1}}{R^{n-2j}}\phi'\equiv \alpha^{n-2k_1}Q_{k_1,j}.
 \end{align*}
 In this way, from \eqref{ers} and proceeding as we did to obtain \eqref{mayorj},
 \begin{align}
 \langle \psi_tf,A_1f\rangle&=\sum_{k_1=0}^k \langle [2k_1+1, n-(2k_1+1)]f,\psi_tf\rangle\notag
 \\
 &=\sum_{k_1=0}^k \alpha^{n-2k_1}\sum_{j=0}^{k_1}\int Q_{k_1,j}(\partial^j_xf)^2=\sum_{j=0}^{k}\sum_{k_1=j}^k\alpha^{n-2k_1}\int Q_{k_1,j}(\partial^j_xf)^2\notag\\
 &=\sum_{j=0}^{k}\alpha^{n-2j}\int Q_{j,j}(\partial^j_xf)^2+\sum_{j=0}^k\sum_{k_1=j+1}^k\alpha^{n-2k_1}\int Q_{k_1,j}(\partial^j_xf)^2\notag
 \\&=\sum_{j=0}^{k}c_{j,j}\frac{\alpha^{n-2j}}{R^{n-2j}}\int \varphi^{n-2j-1}\phi'(\partial^j_xf)^2+\sum_{j=0}^k\sum_{k_1=j+1}^kc_{k_1,j}\frac{\alpha^{n-2k_1}}{R^{n-2j}}\int \varphi^{n-4k_1+2j-1}\phi'(\partial^j_xf)^2.\label{psisubt}
 \end{align}
 To estimate the term $\langle S_1f,\partial_t f\rangle $ in \eqref{suno} we notice that, since $S_1$ is the sum of operators of the form $T_1,\cdots T_n$ as described above,  from the product rule for differentiation we see that $\partial_t[(S_1f)]=S_{1t}f+ S_1\partial_t f$, where $S_{1t}$ is  a sum of compositions of the form $Q_1Q_2\cdots Q_n$ where one of the $Q_i's$ is $B_t$ and the others are either $B$ or $\partial_x$. In this way, by applying  integraton by parts we conclude that
$$\langle S_1f   ,  \partial_tf \rangle=-\langle S_{1t}f+S_1\partial_tf,f\rangle=-\langle S_{1t}f,f\rangle-\langle \partial_tf, S_1f\rangle.$$
Therefore $\langle S_1f   ,  \partial_tf \rangle=-\frac12\langle S_{1t}f,f\rangle$.
Proceeding as we did to obtain \eqref{ers} to \eqref{psisubt} we see that  $\langle S_{1t}f,f\rangle$ has the same form as the right hand side of \eqref{psisubt}.

To conclude the computation of \eqref{suno} we see that
\begin{equation}\langle \psi_t f,\partial_t f\rangle=-\frac12\int\psi_{tt}f^2=-\alpha\int ((\phi')^2+\varphi\phi'')f^2.\label{fini}
\end{equation}
We return to \eqref{suno} and compare the terms of the form $\frac{\alpha^N}{R^M}$ (for diverse integer powers $N$ and $M$) in  \eqref{IjmasIIj},  \eqref{psisubt}, and \eqref{fini}, especially,   we compare the highest order terms 
$$\frac{\alpha^{2n-2j-1}}{R^{2n-2j}}\quad \text{(in \eqref{IjmasIIj})}\quad\text{and} \quad \frac{\alpha^{n-2j}}{R^{n-2j}}\quad\text{(in \eqref{psisubt})}.
$$ 
Since from the hypotheses of the lemma,   $1\leq\varphi\leq5$ in the support of $f$, we conclude that there is a constant $\overline C=\overline C(n, \|\phi'\|_{L^\infty}+\|\phi''\|_{L^\infty})>1$ such that if we take $\alpha>\overline CR^{n/(n-1)}$, then, the  leading terms  in \eqref{IjmasIIj} with coefficient $\frac12 nA_j^n(2\alpha)^{2n-2j-1}/R^{2n-2j}$ absorb the other terms appearing in \eqref{IjmasIIj},  as well as all terms in \eqref{psisubt}, and \eqref{fini}. Therefore, from \eqref{suno} and \eqref{Te} we have that
\begin{equation}
\|Tf\|^2\geq 2\langle Sf,Af\rangle\geq C\sum_{j=0}^{n-1} A^n_j \frac{\alpha^{2n-2j-1}}{R^{2n-2j}}\int(\partial_x^jf)^2,\label{SfAf}
\end{equation}
where $C$ depends only upon $n$.

To prove \eqref{dxmb} we must obtain an expression similar to \eqref{SfAf} with the integrals $\int(\partial_x^jf)^2$ replaced by $\int((\partial_x+B)^jf)^2$.
To do that, using integration by parts,  we observe that for $m=1,\cdots,n-1$, 
\begin{align}
\int((\partial_x+B)^mf)^2&=\int (\partial_x^mf)^2+\sum_{j=0}^{m-1}\sum_{\substack{r,s\geq 0\\r+2s+2j=2m}}c_{m,s,j}\int B^rB_x^s(\partial_x^jf)^2\notag\\
&=\int (\partial_x^mf)^2+\sum_{j=0}^{m-1}\sum_{\substack{r,s\geq0\\r+2s+2j=2m}}c_{m,s,j} \frac{(2\alpha)^{r+s}}{R^{r+2s}}\int\varphi^r(\partial_x^jf)^2\notag\\
&=\int (\partial_x^mf)^2+\sum_{j=0}^{m-1}\sum_{s=0}^{m-j}c_{m,s,j}\frac{(2\alpha)^{2m-2j-s}}{R^{2m-2j}}\int\varphi^{2m-2s-2j}(\partial_x^jf)^2\,.
\label{cambio}
\end{align}
Since $\alpha>\overline CR^{n/(n-1)}>1$, and in the support of $f$, $1\leq\varphi\leq5$,  from \eqref{cambio} it follows that
\begin{equation}\int((\partial_x+B)^mf)^2\leq \int (\partial_x^mf)^2+C' \sum_{j=0}^{m-1}\frac{\alpha^{2m-2j}}{R^{2m-2j}}\int(\partial^j_xf)^2,
\label{ecu9}\end{equation}
where $C'>0$ depends only on $n$.
Now, from  \eqref{SfAf} we see that, if $K_{n-1}$ is a constant with $0<K_{n-1}\leq A^n_{n-1}=A^n_0=n$, then
\begin{equation*}
\|Tf\|^2\geq CK_{n-1}\frac{\alpha}{R^2} \int(\partial^{n-1}_xf)^2+ C\sum_{j=0}^{n-2} A^n_j \frac{\alpha^{2n-2j-1}}{R^{2n-2j}}\int(\partial_x^jf)^2,
\end{equation*}
Thus, applying \eqref{ecu9} with $m=n-1$, we conclude that
\begin{align*}
\|Tf\|^2&\geq CK_{n-1}\frac{\alpha}{R^2}\left[ \int((\partial_x+B)^{n-1}f)^2- C' \sum_{j=0}^{n-2}\frac{\alpha^{2n-2-2j}}{R^{2n-2-2j}}
\int(\partial^j_xf)^2\right]
+C\sum_{j=0}^{n-2} A^n_j \frac{\alpha^{2n-2j-1}}{R^{2n-2j}}\int(\partial_x^jf)^2\\
&=CK_{n-1}\frac{\alpha}{R^2} \int((\partial_x+B)^{n-1}f)^2+C\sum_{j=0}^{n-2} (A^n_j -K_{n-1}C')\frac{\alpha^{2n-2j-1}}{R^{2n-2j}}\int(\partial_x^jf)^2.
\end{align*}
Therefore, by choosing $K_{n-1}=\text{min}\,\{A_{n-1}^n, \frac 12A_{n-2}^n/C',\cdots, \frac 12A_{0}^n/C'\}$, we obtain that 
\begin{equation}
\|Tf\|^2\geq C K_{n-1}\frac{\alpha}{R^2} \int((\partial_x+B)^{n-1}f)^2+C\sum_{j=0}^{n-2} \frac 12 A^n_j \frac{\alpha^{2n-2j-1}}{R^{2n-2j}}\int(\partial_x^jf)^2,\label{repla}\end{equation}
and in this way we obtain an expression similar to \eqref{SfAf} with the first integral term $\int(\partial^{n-1}_xf)^2$ replaced by $\int((\partial_x+B)^{n-1})f)^2$. Notice that from \eqref{comb1}, $K_{n-1}>0$. Proceeding in a similar manner, using \eqref{repla}, succesively applying \eqref{ecu9} with $m=n-2,n-3,\cdots, 1$, and taking adequate values of $K_{n-2},\cdots, K_1>0$, we can perform the replacement of the remaining integrals. Since the $\min\{K_1,\cdots,K_{n-1}\}>0$,  \eqref{dxmb} follows and the  proof of  Lemma \ref{lema3} is complete.
 \end{proof}
 
 We now prove that all the coefficients $A_j^n$ defined in \eqref{Asubnj2} are positive.
  \begin{lemma}\label{lemaaj}
 For integer $k\geq 1$, let $n=2k+1$. Then, for  $j=0,\cdots,n-1$,
\begin{equation}
  A^n_j:=(-1)^{j}\sum_{\substack{k_1+k_2=j\\ 0\leq k_1,k_2\leq k}}
 (2k_2+1-2k_1){\binom n{2k_1}}{\binom n{ 2k_2+1}}=n{\binom{n-1}j}.\label{anjanj}
\end{equation}
\end{lemma}
 \begin{proof}
 Using the formula $(1+x)^n=\sum_{r=0}^n{\binom nr}x^r$ we see that for $x,\beta\in\mathbb R$, $x\not=0$, 
 \begin{align*}h_\beta(x):&=\frac14\Bigl((1+\beta x)^n+(1-\beta x)^n\Bigr)\Bigl((1+\beta/ x)^n-(1-\beta/ x)^n\Bigr)\\
 &=\sum_{k_1,k_2=0}^k
 {\textstyle{\binom n{2k_1}}}{\textstyle{\binom n{2k_2+1}}}\beta^{2k_1+2k_2+1}x^{2k_1-2k_2-1}\\&=\sum_{j=0}^{2k}\sum_{\substack{k_1+k_2=j\\ 0\leq k_1,k_2\leq k}}{\textstyle{\binom n{2k_1}}{\binom n{2k_2+1}}}\beta^{2j+1}x^{2k_1-2k_2-1}.
 \end{align*}
Therefore
 \begin{equation}h'_\beta(1)=\sum_{j=0}^{2k}\beta^{2j+1}\kern-10pt\sum_{\substack{k_1+k_2=j\\ 0\leq k_1,k_2\leq k}}\kern-10pt(2k_1-2k_2-1){\textstyle{\binom n{2k_1}}{\binom n{2k_2+1}}}=\sum_{j=0}^{2k}(-1)^{j+1}A_j^n\beta^{2j+1}.\label{expaj}
\end{equation}
But
\begin{align*}
\frac4{n\beta}h'_\beta(x)&= \Bigl( (1+\beta x)^{n-1}-(1-\beta x)^{n-1}\Bigr)\Bigl( (1+\beta/x)^n-(1-\beta/x)^n\Bigr)\\
&+\Bigl( (1+\beta x)^n+(1-\beta x)^n   \Bigr) \Bigl((1+\beta/x)^{n-1}+(1-\beta/x)^{n-1}   \Bigr)\Bigl(-\frac 1{x^2}\Bigr).
\end{align*}
Therefore, with $x=1$,
\begin{align*}
\frac4{n\beta}h'_\beta(1)&= \Bigl( (1+\beta )^{n-1}-(1-\beta )^{n-1}\Bigr)\Bigl( (1+\beta)^n-(1-\beta)^n\Bigr)\\
&-\Bigl( (1+\beta )^n+(1-\beta )^n   \Bigr) \Bigl((1+\beta)^{n-1}+(1-\beta)^{n-1}   \Bigr)\\
&= -2(1+\beta)^{n-1}(1-\beta)^n-2(1+\beta)^n(1-\beta)^{n-1}=-4(1-\beta^2)^{n-1}.
\end{align*}
In this way,
\begin{equation*}
h'_\beta(1)=\sum_{j=0}^{n-1}(-1)^{j+1}n{\binom{n-1}j}\beta^{2j+1},
\end{equation*}
which, with together with \eqref{expaj}, gives \eqref{anjanj}. \end{proof}
 In the following
 theorem we apply Lemma \ref{lema3} to the difference $w$ of two solutions of equation \eqref{ec} to  obtain a bound of the $L^2$-norm of $w$ in a small rectangule with the $H^{n-1}$ norm of $w$ in a distant rectangle $[R-1,R]\times[0,1]$.

 \begin{Theorem}\label{estinferior}
 Let $\e>0$. For $r\in(0,1)$, $R>2$ and $u_1,u_2\in C([0,1];H^n(\mathbb R))$ solutions of equation \eqref{ec} define $w=u_1-u_2$, $Q=[0,1]\times[r,1-r]$, and
 \begin{equation}
 A_R(w)=\int_0^1\int_{R}^{R+1 }\sum_{j=0}^{n-1}(\partial_x^jw)^2\,dx\,dt\,.\label{Aw}
 \end{equation}
 Suppose that $\|w\|_{L^2(Q)}\geq\delta>0$. Then there exist constants $C>0$,   $C_*=C_*(n,r)>0$, and $R_0>1$ such that

 \begin{equation}
 \|w\|_{L^2(Q)}\leq Ce^{C_* R^\gamma}A_R(w)\, \quad\text{for all }R>R_0\,,\label{hd}
 \end{equation}
 where, 
 \begin{equation}
 \gamma\equiv\gamma_{n,p}=\begin{cases} \frac{n}{n-1}&\text{if }\,p=0,\cdots,k=\frac{n-1}2,\\
                                                  \\                     \frac{2(n-p)}{2(n-p)-1}+\epsilon&\text{if }\,p=k+1,\cdots,n-1.
                                                 \end{cases}\label{inf12}
                                                 \end{equation}
 
 \end{Theorem}
 Notice that for the KdV Hierarchy, $p=n-2$,  and thus we have an exponential $e^{C_* R^{4/3_+}}$ in \eqref{hd}.
\begin{proof} 
From \eqref{ec}, by a procedure similar to that used to obtain \eqref{Ad}, it can be seen that
$w$ is a solution of the equation
\begin{equation}\partial_tw+(-1)^{k+1}\partial^n_xw=-[P(z_1)-P(z_2)]=-\sum_{j=0}^pF_j\partial_x^jw \,,\label{Fjj}
\end{equation}
where  each $F_j$ is a polynomial in $\partial_x^{j_1}u_1$ and $\partial_x^{j_2}u_2$ with $j_1,j_2\leq p\leq n-1$. From the embedding of $H^1(\mathbb R)$ in $L^\infty(\mathbb R)$,    the functions $F_j=F_j(x,t)$  are bounded  in  $D=\mathbb R\times[0,1]$.
Let $\phi\in C^\infty([0,1])$ be a function such that $\phi\equiv 0$ in $[0,\frac r2]\cup[1-\frac r2,1]$, $\phi\equiv 4$ in $[r,1-r]$, $\phi$ increasing in $[\frac r2,r]$, and decreasing in $[1-r, 1-\frac r2]$. Let us choose functions $\widetilde\mu,\, \widetilde\theta\in C^\infty(\mathbb R)$, such that $\widetilde\mu\equiv0$ in $(-\infty,1]$, $\widetilde\mu\equiv1$ in $[2,\infty)$,  $\widetilde\theta\equiv1$ in $(-\infty,0]$,  $\widetilde\theta\equiv0$ in $[1,\infty)$, and define $\mu_R(x,t)\equiv\mu(x,t):=\widetilde\mu(\frac xR+\phi(t))$ and $\theta_R(x)\equiv\theta(x):=\widetilde\theta(x-R)$.
Let $g:=\mu\theta w$. Then, it can be seen that in the support of $\mu\theta$,   $1\leq\frac xR+\phi(t)\leq5$. Besides $g(0)=g(1)=0$. Thus $g$ satisfies the hypotheses of Lemma \ref{lema3}. With $\psi=\alpha(\frac xR+\phi(t))^2$, $\alpha>0$,  as in the statement of Lemma \ref{lema3}, we compute
\begin{align*}
e^\psi(\partial_tg&+(-1)^{k+1}\partial^n_xg)= e^\psi\Bigl(\mu_t\theta w+\mu\theta\partial_t w+\mu\theta(-1)^{k+1}\partial^n_xw
+\sum_{r=1}^n c_{n,r}\partial_x^r(\mu\theta)\partial_x^{n-r}w\Bigr)\\
&=e^\psi\Bigl(-\sum_{j=0}^pF_j\mu\theta\partial^j_xw+\mu_t\theta w + \sum_{r=1}^nc_{n,r}\mu\partial_x^r\theta\partial_x^{n-r}w+\sum_{r=1}^n\sum_{s=1}^rc_{n,r}c_{r,s}\partial_x^s\mu\partial_x^{r-s}\theta\partial_x^{n-r}w\Bigr)\\
&=e^\psi\Bigl(-\sum_{j=0}^pF_j\partial_x^j(\mu\theta w)+\sum_{j=1}^pF_j\sum_{r=1}^jc_{j,r}\partial_x^r(\mu\theta)\partial_x^{j-r}w \\
&\quad\quad\quad+\mu_t\theta w+\mu \sum_{r=1}^nc_{n,r}\partial_x^r\theta\partial_x^{n-r}w+\sum_{r=1}^n\sum_{s=1}^rc_{n,r}c_{r,s}\partial_x^s\mu\partial_x^{r-s}\theta\partial_x^{n-r}w\Bigr)\\
&=  e^\psi\Bigl(-\sum_{j=0}^pF_j\partial_x^jg+\mu\sum_{j=1}^pF_j\sum_{r=1}^jc_{j,r}\partial_x^r\theta\partial_x^{j-r}w+\sum_{j=1}^p\sum_{r=1}^j\sum_{s=1}^rF_jc_{j,r}c_{r,s}\partial_x^s\mu\partial_x^{r-s}\theta\partial_x^{j-r}w\\
&\quad\quad\quad+\mu_t\theta w+\mu \sum_{r=1}^nc_{n,r}\partial_x^r\theta\partial_x^{n-r}w+\sum_{r=1}^n\sum_{s=1}^rc_{n,r}c_{r,s}\partial_x^s\mu\partial_x^{r-s}\theta\partial_x^{n-r}w\Bigr).\\
\end{align*}
To obtain a bound for the $L^2$-norm of the right-hand side of the former expression we take into account the following facts: The derivatives $\partial_x^s\mu$ ($s\geq 1$) are supported in $\subseteq
  \{(x,t)\mid 1<\frac xR+\phi(t)<2\}$ and  thus $ e^\psi\leq e^{4\alpha}$ in this set whose area is of order $R$. Also, $\partial_x^r\theta$ ($r\geq 1$)  is supported in $[R,R+1]\times[0,1]$, and  $e^\psi\leq e^{25\alpha}$  in this rectangle. Besides, $w$, the functions $F_j$, and all the derivatives of $\mu$, $\theta$, and $w$,  are bounded by a constant independent of $R$.   From this considerations, and applying Lemma \ref{lema3} we obtain that
\begin{align}
\sum_{j=0}^{n-1}\frac{\alpha^{n-j-1/2}}{R^{n-j}}\|e^\psi\partial_x^jg\|&\leq  C\|e^\psi(\partial_t+(-1)^{k+1}\partial_x^n)g\|\notag\\
&\leq C\sum_{j=0}^p\|e^\psi\partial_x^jg\|+Ce^{25\alpha}\sum_{j=0}^{n-1}\|\partial^j_xw\|_{L^2([R,R+1]\times[0,1])}
+CR^{1/2}e^{4\alpha}.\label{inf9}\end{align}
Let $\overline C$ be the constant in Lemma \ref{lema3}. Then $\overline C=\overline C(n,r)$. If we take $\alpha=(1+\overline C)R^{1+s}$, with $s\geq\frac1{n-1}$, then 
$\alpha>\overline CR^{\frac {n}{n-1}}$, and for $j=1,\cdots,p$
$$ 
\frac{\alpha^{n-j-1/2}}{R^{n-j}}=(1+\overline C)^{n-j-\frac12}R^{-\frac12+s(n-j-\frac12)}\geq R^{s(n-p-\frac12)-\frac12} $$

and therefore, after discarding the terms with $j>p$ on the left-hand side of \eqref{inf9}, and bearing in mind the definition of $A_R(w)$ given in \eqref{Aw}, we have that
 \begin{equation}
\sum_{j=0}^{p}R^{s(n-p-\frac12)-\frac12}\|e^\psi\partial_x^jg\|\leq C\sum_{j=0}^p\|e^\psi\partial_x^jg\|+Ce^{25\alpha}A_R(w)+CR^{1/2}e^{4\alpha}.\label{inf10}\end{equation}
We will choose $s\geq \frac1{n-1}$ in such a way that $s(n-p-1/2)-1/2>0$, that is, $s>1/(2n-2p-1)$. Then, by making $R$ sufficiently large we can make the left-hand side of \eqref{inf10} more than twice the first term on the right-hand side, allowing the absortion of this last term to obtain that 
\begin{equation} \sum_{j=0}^{p}R^{s(n-p-\frac12)-\frac12}\|e^\psi\partial_x^jg\|\leq Ce^{25\alpha}A_R(w)+CR^{1/2}e^{4\alpha}.\label{inf4}\end{equation}

 To choose the appropriate value of $s$ we see that if $p\leq{(n-1)}/2=k$, then $$\frac1{(2n-2p-1)}\leq \frac1n<\frac1{n-1},$$ and we choose $s=1/(n-1)$. Then with $\alpha=(1+\overline C)R^{1+s}=(1+\overline C)R^{(n-1)/n}$, we have \eqref{inf4} for large $R$.
 
 If $(n-1)/2\leq  p\leq n-1$, that is if $ {(n+1)}/2\leq p\leq n-1$, then, with $\epsilon>0$ and $s=1/(2n-2p-1)+\epsilon$, that is, with $\alpha=(1+\overline C)R^{\frac{2(n-p)}{2(n-p)-1}+\epsilon}$, we obtain \eqref{inf4} for large $R$.
  
 Since $\frac xR+\phi(t)\geq 4$ in $Q:=[0,1]\times[r,1-r]$ and $\mu\equiv1$, and $\theta\equiv1$ in $Q$, we can replace the left-hand side of \eqref{inf4} by a smaller amount to conclude that for $R$ sufficiently large
\begin{equation} e^{16\alpha}\|w\|_{L^2(Q)}\leq Ce^{25\alpha}A_R(w)+CR^{1/2}e^{4\alpha},\label{inf11}\end{equation} 
Hence, with $\gamma$ as in \eqref{inf12}, and $\alpha=(1+\overline C)R^\gamma$,
\begin{equation*} e^{16(1+\overline C)R^\gamma}\|w\|_{L^2(Q)}\leq Ce^{25(1+\overline C)R^\gamma}A_R(w)+CR^{1/2}e^{4(1+\overline C)R^\gamma},\end{equation*} 
 Since $\|w\|_{L^2(Q)}\geq\delta>0$, by making $R$ sufficiently large we can absorb the last term on the right-hand side of the former inequality with the left-hand side to obtain \eqref{hd} with $C_*=9(1+\overline C)$, which completes the proof of Theorem \ref{estinferior}.
 \end{proof}
 \section{Proofs of Theorem I and Theorem II}\label{sectionmain}

 For Theorem I we present a proof  which can be adapted with minor changes to prove Theorem II.
 
 {\it Proof of Theorem I.}
 
 Since from hypothesis \eqref{w},  $e^{x_+^{4/3+\epsilon}}w(0) \in L^2(\mathbb R)$,  it follows that $\|e^{ax_+^{4/3+\epsilon/2}}w(0)\|_{\ldos}\leq C_a<\infty$ for all $a>0$. The same property holds for $w(1)$. Also, by an interpolation argument similar to that in \eqref{inter} 
 \begin{equation}
 \| e^{ax_+^{4/3+\epsilon/2}}\partial^j_xw(i)\|_{\ldos}\leq C_a<\infty \quad \text{ for all } a>0, \quad j=1,\cdots,n,\quad i=0,1.\label{aa}
 \end{equation}
 Suppose that $w$ does not vanish identically in $D:=\R\times[0,1]$.  Then,  there is a rectangle $Q:=[x_0,x_0+1]\times[r,1-r]$, for some $x_0\in\mathbb R$ and $r\in(0,1)$, such that $\|w\|_{L^2(Q)}>0$. If we consider  translations $\tilde u_i$ of $u_i$, defined by $\tilde u_i(x,t):=u_i(x+x_0,t)$, $i=1,2$, then, it can be seen that $\tilde u_1$, $\tilde u_2$, and $\tilde w:= \tilde u_1-\tilde u_2$, satisfy the hypotheses of Theorem I. In this way, making a translation if necessary, we can suppose without loss of generality that  $Q=[0,1]\times[r,1-r]$.
 
 Let $\eta\in C^\infty(\R)$ be a function supported in $(0,1)$ and such that  $\int\eta=1$. For $R>1$ and $N>4R+1$, define $\phi_{R,N}(x)\equiv\phi(x):=\int_{-\infty}^x\eta(x'-R)-\eta(x'-N))\,dx'$. Then $\phi_{R,N}=1$ in $[R+1,N]$, $supp\,\phi_{R,N}\subseteq [R,N+1]$ and $|\phi^{(j)}_{R,N}|\leq c_j$ with $c_j$ independent of $R$ and $N$. We will apply Theorem \ref{car} to the function $v_{R,N}\equiv v:=\phi_{R,N}\,w$. From \eqref{Fjj}  with $p=n-2$, $v$ satisfies
 \begin{align}
 \partial_tv+(-1)^{k+1}\partial_x^nv
 &=\phi(\partial_tw+(-1)^{k+1}\partial^n_xw)+\sum_{j=1}^nc_{n,j}\phi^{(j)}\partial_x^{n-j}w\notag\\
 &=-\sum_{j=0}^p\phi F_j\partial_x^jw+\sum_{j=1}^nc_{n,j}\phi^{(j)}\partial_x^{n-j} w\notag\\
 &\quad=-\sum_{j=0}^pF_j\partial_x^j(\phi  w)+\sum_{j=1}^pF_j\sum_{r=1}^j c_{j,r}\phi^{(r)}\partial^{j-r}_xw+\sum_{j=1}^nc_{n,j}\phi^{(j)}\partial_x^{n-j} w.\label{t1}
  \end{align}

For $\lambda>2$ we now apply together \eqref{car1} and \eqref{car2} in Theorem \ref{car} to $v$. We use \eqref{t1}, H\"older's inequailty, and the fact that $\|\cdot\|\ldTldx\leq\|\cdot\|\liTldx$ and $\|\cdot\|\luTldx\leq\|\cdot\|\ldTldx$, and  take into account that $\phi$ is supported in $[R,N+1]$ and its derivatives $\phi^{(j)}$ are supported in $[R,R+1]\cup[N,N+1]$,  to obtain that
 \begin{align}
 \|\el\phi w\|\ldTldx +&\sum_{j=1}^{n-1}\|e^{\lambda x}\partial_x^j(\phi w)\|\lixldT\notag\\
 &\leq C\lambda^{n-1}\| |J^{n}(\el\phi w(1)) | +|J^{n}(\el\phi w(0))|\|_\ldos+C\|e^{\lambda x}(\partial_t+(-1)^{k+1}\partial_x^n)v\|\luTldx{}_\cap{}\luxldT\notag\\
 &\leq C\lambda^{2n-1}\sum_{j=0}^1\|\el(|w(j)|+|\partial^n_xw(j)|)\|_{L^2([R,\infty))}\tag{I}\label{Ia}\\
 &+C\|F_0\|_{L^2_TL^\infty_{x\geq R}}\|\el\phi w\|\ldTldx+C\|F_0\|_{L^2_{x\geq R}L^\infty_T}\|\el\phi w\|\ldxldT\label{IIa}\tag{II}
 \\&+C\sum_{j=1}^p\|F_j\|_{L^2_{x\geq R}L^\infty_T}\|\el\partial_x^j(\phi w)\|\lixldT
 +C\sum_{j=1}^p\|F_j\|_{L^1_{x\geq R}L^\infty_T}\|\el\partial_x^j(\phi w)\|\lixldT\label{IIIa}\tag{III}\\&
 +{C}e^{\lambda (R+1)}\sum_{j=0}^{n-1}\|\partial_x^jw\|\liTldx+ {C}\sum_{j=0}^{n-1}e^{\lambda{(N+1)}}\|\partial_x^jw\|_{L^\infty_TL^2_{x\geq N}}\tag{IV}\label{IVa}\\&=\text{I}+\text{II}+\text{III}+\text{IV}.
 \label{F0}
\end{align}
 
 where ${C}$  does not depend on $\lambda$, $N$, and $R$.
 
Taking into account the specific form of the polynomial $P$ in \eqref{poli1},  since $\sum_{j=0}^{n-1}F_j\partial_x^jw=P(z_1)-P(z_2)$, we see from  \eqref{P}, and \eqref{Ad} that
each $F_j$ is a polynomial in $\partial_x^{j_1}u_1$ and $\partial_x^{j_2}u_2$ with $j_1,j_2\leq n-3$, except for $F_0$ which has a term $\partial_x^{n-2}u_1$ comming from the quadratic term $w\partial_x^{n-2}u_1$ in  $A_2(z_1)-A_2(z_2)$ (see \eqref{Idos}). 
 From \eqref{alpfa1} it follows that
\begin{equation}\|(1+x_+)^{ \alpha(1-\frac {l+1}{n+1})}\partial_x^l u_i(t)\|_{L^\infty(\mathbb R)}\leq C\,,\quad\text{for all }t\in[0,1] \;\label{interinter}\text{and }l=0,\cdots,n.
\end{equation} 
For $l\leq n-3$, $\alpha(1-\frac {j+1}{n+1})>\frac{n+1}3(1-\frac{n-2}{n+1})=1$. 
Therefore $\|(1+x_+)^{1^+}F_j\|_{L^\infty_TL^\infty_x}<\infty$ for $j=1,\cdots,p$. In a similar way, with $l=n-2$ in \eqref{interinter}, we see that  $\|(1+x_+)^{2/3}F_0\|_{L^\infty_TL^\infty_x}<\infty$. From this decay of the functions $F_j$ we conclude that, by taking $R$ sufficiently large, the norms involving these functions in  II and III  of \eqref{F0} can be made small in such a way that II and III can be absorved by the left-hand side of \eqref{F0}.

After we perform this absortion, and taking into account that  $\phi\equiv 1$ in $[4R,4R+1]$,  we replace the left-hand side of \eqref{F0} by a smaller amount to obtain that
\begin{align}
 \|e^{ \lambda x}w\|_{L^2([4R,4R+1]\times[0,1])} +&\sum_{j=1}^{n-1}\|e^{\lambda x}\partial_x^jw\|_{L^2([4R,4R+1]\times[0,1])}\notag\\
 &\leq C\lambda^{2n-1}\sum_{j=0}^1\|\el(|w(j)|+|\partial^n_xw(j))|\|_{L^2([R,\infty))}\tag{I}\\
 &
 + {C}e^{\lambda (R+1)}\sum_{j=0}^{n-1}\|\partial_x^jw\|\liTldx+  {C}e^{\lambda{(N+1)}}\|\partial_x^jw\|_{L^\infty_TL^2_{x\geq N}}\tag{IV}\\&=\text{I}+\text{IV}.
 \label{F1}
\end{align}

From the decay hypothesis \eqref{w} of $w$  and the exponential decay preservation proved in Theorem \ref{decaimiento}, it follows that    $\|e^{\lambda x}w(t)\|_\ldos\leq C_\lambda<\infty$, for all $\lambda>0$ and all $t\in[0,1]$. From an interpolation argument similar to that in \eqref{inter}, we also have that  $\|e^{\lambda x}\partial_x^jw(t)\|_\ldos\leq C_\lambda$, for $j=1,\cdots,n$. Therefore, 

\begin{equation*}
\text{IV}\leq   {C}e^{\lambda (R+1)}\sum_{j=0}^{n-1}\|\partial_x^jw\|\liTldx + {C}e^{\lambda (N+1)}e^{-2\lambda N}\|e^{2\lambda x}\partial_x^jw\|_{L^\infty_TL^2_x}\leq   {C}e^{\lambda (R+1)}+C_\lambda e^{-\lambda(N-1)}\,,\label{IV}
\end{equation*}
Then, if we make $N\to\infty$, from \eqref{F1} we conclude that
\begin{equation}
e^{4R\lambda}\sum_{j=0}^{n-1}\|\partial_x^jw\|_{L^2([4R,4R+1]\times[0,1])}\leq C\lambda^{2n-1}\sum_{j=0}^1\|\el(|w(j)|+|\partial^n_xw(j))|\|_{L^2([R,\infty))}+{C}e^{\lambda (R+1)}.\label{4R}
\end{equation}
For $a>0$  to be determined later, we take $\lambda=aR^{1/3+\epsilon/2}$. Since   $\lambda x=aR^{1/3+\epsilon/2} x\leq ax^{4/3+\epsilon/2}$ for $x\geq R$, we have from \eqref{4R} that
\begin{align*}
e^{4aR^{4/3+\epsilon/2}}&\sum_{j=0}^{n-1}\|\partial_x^jw\|_{L^2([4R,4R+1]\times[0,1])}
\\ &\leq Ca^{2n-1}R^{(2n-1)(\frac13+\epsilon/2)}\sum_{j=0}^1\|e^{ax^{4/3+\epsilon/2}}(|w(j)|+|\partial^n_xw(j))|\|_{L^2(\R)}+{C}e^{aR^{1/3+\epsilon/2} (R+1)},\end{align*}
and thus, from the definition of $A_R(w)$ given in \eqref{Aw} and from \eqref{aa}
\begin{equation}
e^{4aR^{4/3+\epsilon/2}}A_R(w)\leq C_aR^{(2n-1)(1/3+\epsilon/2)}+{C}e^{2a R^{4/3+\epsilon/2}}\leq C_ae^{2a R^{4/3+\epsilon/2}}.
\label{5R}
\end{equation}
We now apply Theorem \ref{estinferior} with $p=n-2$ to obtain that
\begin{equation}
\|w\|_{L^2(Q)}\leq Ce^{C_*R^{4/3+\epsilon/2}}A_R(w) \leq Ce^{C_*R^{4/3+\epsilon/2}}C_ae^{-2aR^{4/3+\epsilon/2}}=C_ae^{(C_*-2a)R^{4/3+\epsilon/2}},
\end{equation}
Where $C^*=C^*(r)$. 
If we fix $a>\frac{C_*}2$, then by taking $R\to\infty$ we conclude that $\|w\|_{L^2(Q)}=0$, which contradicts the original assumption $\|w\|_{L^2(Q)}\not=0$. Then we conclude that $w\equiv0$, and Theorem I is proved.
\qed
\vskip10pt
{\it Proof of Theorem II}.

From Remark \ref{remdec}, $w$ satisfies \eqref{tres1}. With $\beta=1$, after applying Gronwall's inequality and taking $N\to\infty$, we can conclude that for $t_0\in[0,1]$ 
 \begin{equation} \int e^xw(t)^2\leq C\int e^xw(t_0)^2\quad \text{for all }t\in[t_0,1].
 \end{equation}
  By making the change of variables $x\mapsto-x$ and $t\mapsto1-t$, and taking into account that $w\in C([0,1]; H^{n+1}(\R)\cap L^2((1+x_-)^{2\alpha_0}\,dx)$ we can also see that
  \begin{equation} \int e^{-x}w(t)^2\leq C\int e^{-x}w(t_0)^2\quad \text{for all }t\in[0,t_0].
 \end{equation}
 Thus we can conclude that if $w(t_0)=0$, then $w\equiv0$.
 
 We will find a constant $a>0$ such that if $w(0), w(1)\in L^2(e^{ax^{n/(n-1)}})$, then $w\equiv 0$. We reason by contradiction. Suppose that $w$ does not vanish identically in $D:=\R\times[0,1]$. Then, by the uniqueness argument just given, $w$ does not vanish identically in $D_0:=\R\times[1/3,2/3]$. Therefore, there is a rectangle $Q:=[x_0,x_0+1]\times[1/3,2/3]$ such that $\|w\|_{L^2(Q)}>0$. By making a translation if necessary, we can suppose without loss of generality that  $Q=[0,1]\times[1/3,2/3]$.
 We now continue applying the same arguments used to prove Theorem I, using $\lambda=aR^{1/(n-1)}$ instead of $aR^{1/3+\epsilon/2}$.  In this case we apply Theorem \ref{estinferior} with $p\leq k$, and $C_*=C_*(1/3)$ and choose $a=\frac{C_*}2+1>\frac{C_*}2$, which gives a value of $a$ which depends only on $n$.

\end{document}